\documentclass[12pt,a4paper]{article}
\usepackage[T1]{fontenc}
\usepackage[utf8]{inputenc}
\usepackage{color}
\usepackage{colortbl}
\usepackage{mathtools}
\usepackage{enumitem}
\usepackage{amsmath}
\usepackage{amsthm}
\usepackage{amssymb}
\usepackage[pdfusetitle,
 bookmarks=true,bookmarksnumbered=false,bookmarksopen=false,
 breaklinks=false,pdfborder={0 0 1},backref=false,colorlinks=false]
 {hyperref}

\makeatletter

\providecommand{\tabularnewline}{\\}

\@ifundefined{date}{}{\date{}}
\usepackage{amsthm}\usepackage{tikz}
\usepackage{graphicx}
\usepackage{array}

\newtheorem{theorem}{Theorem}[section]
\newtheorem{lemma}[theorem]{Lemma}\newtheorem{proposition}[theorem]{Proposition}\newtheorem{corollary}[theorem]{Corollary}\theoremstyle{definition}
\newtheorem{definition}[theorem]{Definition}\newtheorem{example}[theorem]{Example}\theoremstyle{remark}
\newtheorem{remark}[theorem]{Remark}

\setlist[enumerate,1]{label={(\arabic*)}, align=left, left=0pt}

\makeatother

\theoremstyle{plain}

\providecommand{\theoremname}{Theorem}

\begin{document}
\title{Priestley Representation of Distributive Precontact Lattices}
% Use  for an abbreviated version of
\author{Sergio A. Celani\thanks{Departamento de Matem\'atica, Facultad de Ciencias Exactas, Universidad Nacional del Centro and CONICET, Pinto 399, CP 7000 Tandil, Argentina. Email: \texttt{scelani@exa.unicen.edu.ar}} \and Luciana Valenzuela\thanks{Same address. Email: \texttt{luvalenzuelaj@gmail.com}}}
% Use  for an abbreviated version of


\date{}
\maketitle

\vspace{1cm}
\begin{center}
\itshape
To Hilary Priestley, whose pioneering work in lattice theory and duality\\
continues to inspire generations of researchers.
\end{center}
\vspace{1cm}

\begin{abstract}
It is well known that, in Boolean algebras, the notions of precontact
relation, quasi-modal operator, and subordination relation are interdefinable.
In contrast, within the setting of distributive lattices, this equivalence
holds only between quasi-modal operators and subordination relations,
but not with precontact relations. In this paper, we study the class
of bounded distributive lattices equipped with a precontact relation,
referred to as precontact lattices. We also examine how conditions
imposed on the precontact relation correspond to first-order properties
in the Priestley dual. In addition, we characterize precontact substructures
and strong precontact sublattices of precontact lattices in terms
of a suitable lattice preorder. Finally, we introduce a notion of
precontact congruence and identify the corresponding closed sets in
the dual space. 
\end{abstract}

\section{Introduction}

The concept of Boolean contact algebras originates from attempts to
provide a region-based alternative to the classical point-based understanding
of space. Motivated by philosophical ideas and later developments
in topology and logic, Boolean contact algebras offer a formal framework
for reasoning about spatial regions and their relationships, especially
about the intuitive notion of ``being in contact''. A contact algebra
extends a Boolean algebra with an additional binary relation intended
to capture the idea that two non-empty regions ``touch'' each other.
This region-based approach avoids reference to points and instead
focuses on the mereological (part-whole) and topological interactions
between spatial entities. The formal study of contact algebras has
been significantly influenced by region-based topology, modal logic,
and qualitative spatial reasoning.

Precontact relations \cite{Dimov1,Dimov3} are a generalization of
the contact relations studied in \cite{Dimov2}, \cite{Duntsch2},
and \cite{Vakarelov3}, which are used to model the interaction of
regions in a geometric or topological space. As shown in \cite{Bezha1}
and \cite{Celani4}, precontact algebras are dual to Boolean subordination
algebras i.e., pairs $\langle B,\prec\rangle$ where $B$ is a Boolean
algebra and $\prec$ is a binary relation on $B$ called the subordination.
It is proved in \cite{Bezha1} that subordination Boolean algebras
and precontact algebras are equivalent formulations of quasi-modal
algebras introduced in \cite{Celani2} (see also \cite{Celani7} and
\cite{Celani4}). The subsequent development in \cite{Celani3} extends
this approach to bounded distributive lattices, thereby generalizing
the notion of modal operators beyond the Boolean setting.

Since the negation operation is absent from the axioms for a contact
relation, it is natural to consider bounded distributive lattices
equipped with either a precontact or a contact relation. To the best
of our knowledge, the earliest study of contact relations on distributive
lattices is due to Düntsch, MacCaull, Vakarelov, and Winter \cite{Duntsch1}.
Their work explores various topological representations of distributive
lattices with a contact relation, characterizing them as sublattices
of a Boolean algebra of regular closed sets of a topological space.
Further studies can be found in \cite{Ivanova} where Ivanova and
Vakarelov investigate topological representations of distributive
lattices endowed with contact, nontangential inclusion, and dual contact
relations, thereby enriching the expressive power of the language
of distributive contact lattices.

The aim of this paper is to further develop the study of bounded distributive
lattices equipped with a precontact relation. Specifically, we focus
on constructing a topological-relational representation and establishing
a duality based on Priestley spaces enriched with an additional binary
relation. This duality generalizes the well-known dualities for Boolean
algebras with a precontact relation, quasi-modal operators, or subordinations,
as developed in various works with different motivations (see \cite{Bezha1},
\cite{Celani2}, \cite{Dimov1}, \cite{Dimov3}, \cite{Duntsch3}
and \cite{Koppelberg}).

In Section {\ref{sec:Precontact-relations-on}} we present several
examples of precontact relations on distributive lattices. We show
that precontact relations are closely connected to a notion of negation,
much like subordinations are related to the concept of a modal operator.
The most common notion of morphism between Boolean contact algebras
is that of a Boolean homomorphism that reflects contact (see, for
example, \cite{Dimov2} or \cite{Gold}). However, when the precontact
relation is modally definable (see \cite{Bezha1}), this definition
does not coincide with that of a modal homomorphism. For this reason,
we will consider two notions of morphism between precontact relations.
The first is the usual notion of lattice homomorphism that reflects
the precontact relation. The second, which we call strong precontact
homomorphism, is designed to coincide---when the precontact relation
is modally definable---with the standard notion of homomorphism between
modal algebras. In fact, in the Boolean setting, when the precontact
relation is modally definable the strong precontact homomorphism coincide
with the usual notion modal homomorphism, as shown in Proposition
\ref{relacion con homo booleano}.

In Section~\ref{sec:Priestley-duality} we focus on the representation
and Priestley-type duality for precontact lattices. We introduce the
notion of precontact space, namely a Priestley space endowed with
a binary relation $R$ that suitably models the precontact relation.
In this setting, we also consider two types of morphisms: stable and
strongly stable morphisms. The notion of strongly stable morphism
was introduced in \cite{Celani2,Celani3} in the context of the representation
theory of quasi-modal algebras and quasi-modal lattices, as a generalization
of the well-known notion of $p$-morphism between Kripke frames. This
framework naturally gives rise to four categories: $\mathbf{PreDL}$
and $\mathbf{PreDLS}$ on the algebraic side, and $\mathbf{PS}$ and
$\mathbf{SPS}$ on the topological side. The categories $\mathbf{PreDL}$
and $\mathbf{PreDLS}$ have the same objects but differ in their morphisms,
and the same holds for $\mathbf{PS}$ and $\mathbf{SPS}$. Finally,
we establish a duality between $\mathbf{PreDL}$ and $\mathbf{PS}$,
and between $\mathbf{PreDLS}$ and $\mathbf{SPS}$.

Following the perspective of correspondence theory in modal logics
and in the theory of precontact algebras \cite{Dimov1,Dimov3}, it
is natural to investigate how certain algebraic properties of precontact
lattices are reflected in the relational structure of their dual spaces.
In Section \ref{sec:Correspondence-theory} we study several classes
of precontact lattices in terms of relational conditions on the associated
precontact spaces. This analysis generalizes known results in the
Boolean setting. In particular, we show that contact lattices correspond
to precontact spaces $\left<X,R\right>$ whose relation $R$ is reflexive
and symmetric. Moreover, we prove that the class of normal lattices
(see Definition \ref{normal and compingent lattice}) corresponds
to precontact spaces where $R$ is an equivalence relation. These
results generalize Theorem 10 from \cite{Duntsch2}, and Theorem 26
from \cite{Celani2}. We also introduce the notion of extensional
precontact lattices as a natural extension of extensional Boolean
algebras (see \cite{Bezha1} and \cite{Dimov2}). Finally, we define
the class of compingent contact lattices and provide a dual characterization
of these structures in terms of an irreducibility condition on the
relation $R$, thereby extending to the setting of precontact lattices
the relational framework developed for Boolean algebras in \cite{Bezha1}.

Section \ref{sec:Subestructures-and-subalgebras} is devoted to studying
certain extensions of the classical notion of subalgebra. In universal
algebra, a subset $B$ of an algebra $A$ (of a given similarity type)
is a subalgebra if and only if the inclusion map $i\colon B\hookrightarrow A$
is a homomorphism. Consequently, any generalization of the notion
of subalgebra is inherently linked to suitable notions of homomorphisms
between precontact lattices. Since in this paper we consider two notions
of morphism between precontact lattices, it is natural to explore
the two corresponding generalizations of the classical concept of
subalgebra. We prove that the notion of a precontact substructure
corresponds to homomorphisms that reflect the precontact relation,
while the notion of a strong precontact sublattice corresponds to
strong precontact homomorphisms. In addition, we characterize both
substructures and strong precontact sublattice in terms of a suitable
lattice preorder. It is worth noting that the notion of substructure
in contact algebras was studied in \cite{Duntsch4}. This notion coincides
with what we call a precontact subalgebra.

In Section \ref{sec:Congruences-of-precontact}, we introduce and
study the notion of precontact congruence -- that is, a congruence
on a distributive lattice that satisfies a compatibility condition
with the precontact relation. We show that this notion is well-behaved
in the sense that it allows for the definition of a precontact relation
on the quotient lattice in such a way that the canonical homomorphism
becomes a strong homomorphism between precontact lattices. Moreover,
when the precontact relation on a Boolean algebra is modally definable,
the notion of precontact congruence coincides with that of a modal
Boolean congruence and when the precontact relation, on a distributive
lattice, is negationally definable it coincides with the notion of
congruence of a lattice with negation.

\section{Preliminaries}

Let $\langle X,\leq\rangle$ be a poset. For each $Y\subseteq X$,
let $[Y)=\{x\in X:\exists y\in Y(y\leq x)\}$ and $(Y]=\{x\in X:\exists y\in Y(x\leq y)\}$.
We say that $Y$ is an upset of $X$ if $Y=[Y)$, and a downset of
$X$ if $Y=(Y]$. If $Y$ is the singleton $\{y\}$, we use the notation
$[y)$ and $(y]$ instead of $[\{y\})$ and $(\{y\}]$, respectively.
We denote by $\mathrm{Up}(X)$ the collection of all upsets of $\langle X,\leq\rangle$.
We also write $\mathcal{P}(X)$ for the power set of $X$. The complement
of a subset $Y\subseteq X$ is denoted by $X-Y$ or $Y^{c}$.

A totally order-disconnected topological space is a triple $\mathbb{X}=\langle X,\leq,\tau\rangle$
where $\langle X,\leq\rangle$ is a poset and $\langle X,\tau\rangle$
is a topological space, and for any $x,y\in X$ with $x\nleq y$,
there exists a clopen upset $U$ such that $x\in U$ and $y\notin U$.
A Priestley space is a compact totally order-disconnected topological
space (see \cite{Priestley1}, \cite{Priestley2}). For a Priestley
space $X$, let $\mathrm{D}(X)$ denote the set of all clopen upsets
of $X$, $\mathrm{OpUp}(X)$ the set of open upsets, and $\mathrm{ClUp}(X)$
the set of closed upsets. If $X$ is a Priestley space, for each $x\in X$
we denote by $\max R(x)=\left\{ y\in R(x):\text{ there is no }z\in R(x)\text{ such that }y<z\right\} $,
where $<$ denotes the strict order induced by the Priestley order
on $X$.

Given a bounded distributive lattice $L=\langle L,\vee,\wedge,0,1\rangle$,
let $\mathrm{X}(L)$ be the set of all prime filters of $L$. Denote
by $\mathrm{Id}(L)$ the collection of all ideals of $L$ and by $\mathrm{Fi}(L)$
the collection of all filters of $L$. The function $\beta:L\rightarrow\mathrm{Up}(\mathrm{X}(L))$
defined by $\beta(a)=\{P\in\mathrm{X}(L):a\in P\}$ is an embedding
of bounded distributive lattices, so that $L\cong\beta\textcolor{blue}{[L]}$.
The structure $\langle\mathrm{X}(L),\subseteq,\tau_{{\rm {X}(L)}}\rangle$
forms a Priestley space under the inclusion order and a topology with
a subbasis of the form $\beta[L]\cup\{\mathrm{X}(L)-\beta(a):a\in L\}$,
giving $L\cong\mathrm{D}(\mathrm{X}(L))=\beta\textcolor{blue}{[L]}$.

The map $\varphi:\mathrm{Id}(L)\rightarrow\mathrm{OpUp}(\mathrm{X}(L))$
given by $\varphi(I)=\{P\in\mathrm{X}(L):P\cap I\neq\emptyset\}$
is a lattice isomorphism between the poset $(\mathrm{Id}(L),\subseteq)$
of ideals of $L$ and the poset $(\mathrm{OpUp}(\mathrm{X}(L)),\subseteq)$
of open upsets of $\mathrm{X}(L)$. If $U$ is an open upset of $\mathrm{X}(L)$,
then $I(U)=\{a\in L:\beta(a)\subseteq U\}$ is the ideal of $L$ corresponding
to $U$. Moreover, for $I$, $J\in\mathrm{Id}(L)$, $I\subseteq J$
iff $\varphi(I)\subseteq\varphi(J)$, $I=I(\varphi(I))$, and $U(I(U))=U$,
for an open upset $U$ of $\mathrm{X}(L)$. Also, the map $\psi:\mathrm{Fi}(L)\rightarrow\mathrm{ClUp}(\mathrm{X}(L))$
defined by $\psi(F)=\{P\in\mathrm{X}(L):F\subseteq P\}$ is a lattice
anti-isomorphism. If $C$ is a closed upset of $\mathrm{X}(L)$, then
$F(C)=\{a\in L:C\subseteq\beta(a)\}$ is the filter of $L$ corresponding
to $C$. For $F,H\in\mathrm{Fi}(L)$, $F\subseteq H$ iff $\psi(H)\subseteq\psi(F)$,
$F=F(\psi(F))$, and $C=C(F(C))$, for a closed upset of $\mathrm{X}(L)$.
It is well known that $\varphi(I)=\bigcup\{\beta(a):a\in I\}$ for
each $I\in\mathrm{Id}(L)$, and $\psi(F)=\bigcap\{\beta(a):a\in F\}$
for each $F\in\mathrm{Fi}(L)$.

For a bounded distributive lattice $L$, the filter (ideal) generated
by a subset $H\subseteq L$ is denoted $\mathrm{Fg}_{L}(H)$ ($\mathrm{Ig}_{L}(H)$).
If $H=\{a\}$, we write $\mathrm{Fg}(a)$ and $\mathrm{Ig}(a)$ instead
of $\mathrm{Fg}(\{a\})$ and $\mathrm{Ig}(\{a\})$, respectively.
A subset $M\subseteq L$ is a sublattice of $L$ if $a\wedge b,\,a\vee b\in M$
for all $a,b\in M$. It is a $(0,1)$-sublattice if $0_{L},1_{L}\in M$. 

A $\Box$-modal lattice is a bounded distributive lattice $L$ expanded
with a unary operation $\Box$, the so-called necessity operator,
satisfying $\Box1=1$ and preserving finite meets, that is, $\Box(a\wedge b)=\Box a\wedge\Box b$
for all $a,b\in L$ (see, for example, \cite{Celani8}). Dually, a
unary operation $\Diamond$ satisfying $\Diamond0=0$ and preserving
finite joins, i.e., $\Diamond(a\vee b)=\Diamond a\vee\Diamond b$
for all $a,b\in L$, is called a possibility operator, and bounded
distributive lattices expanded with such an operation are called $\Diamond$-modal
lattices (see, for example, \cite{CigLafPet,Ge}).

Let $R$ be a binary relation on a set $X$. As usual, for $x\in X$,
let $R(x)=\{y\in X:(x,y)\in R\}$ and $R^{-1}(x)=\{y\in X:(y,x)\in R\}$.
Also, for $V\subseteq X$, let $R(V)=\bigcup\limits_{x\in V}R(x)$
and $R^{-1}(V)=\bigcup\limits_{x\in V}R^{-1}(x)$.

\section{Precontact relations on bounded distributive lattices}

\label{sec:Precontact-relations-on}

In this section, we introduce the notion of a precontact bounded distributive
lattice, present several examples, and investigate its elementary
properties. Precontact algebras were primarily studied in \cite{Dimov1},
\cite{Dimov3}, and \cite{Duntsch2}. In the Boolean setting, a precontact
relation is equivalent to both the notion of subordination \cite{Bezha1}
and that of a quasi-modal operator \cite{Celani2}. However, in the
absence of an involutive negation, these notions no longer coincide.
We also introduce precontact homomorphisms and strong precontact homomorphisms,
extending the standard morphisms between precontact Boolean algebras
to the broader setting of bounded distributive lattices. These homomorphisms
capture different ways of preserving the precontact relation and provide
a natural generalization of homomorphisms for modal algebras.

\begin{definition} A precontact relation on a bounded distributive
lattice $L$ is a binary relation $\mathsf{C}$ satisfying the following
conditions: 
\begin{enumerate}
\item[$\mathrm{(PC1)}$ ] If $(a,b)\in\mathsf{C}$, then $a,b\neq0$. 
\item[$\mathrm{(PC2)}$ ] $(a\vee b,x)\in\mathsf{C}$ if and only if $(a,x)\in\mathsf{C}$
or $(b,x)\in\mathsf{C}$. 
\item[$\mathrm{(PC3)}$ ] $(a,b\vee x)\in\mathsf{C}$ if and only if $(a,b)\in\mathsf{C}$
or $(a,x)\in\mathsf{C}.$ 
\end{enumerate}
A precontact lattice is a pair $\langle L,\mathsf{C}\rangle$, where
$L$ is a bounded distributive lattice and $\mathsf{C}$ is a precontact
relation on $L$. \end{definition}

The class of all precontact lattices is denoted by $\mathbf{PreDL}$.
If $L$ is a Boolean algebra and $\mathsf{C}$ is a precontact relation
on $L$, we say that $\langle L,\mathsf{C}\rangle$ is a precontact
Boolean algebra or proximity Boolean algebra \cite{Dimov1,Duntsch2}.
We also consider precontact lattices satisfying some additional axioms: 
\begin{enumerate}
\item[(PC4)] Axiom of reflexivity: If $a\neq0$ then $(a,a)\in\mathsf{C}$. 
\item[(PC5)] Axiom of symmetry: If $(a,b)\in\mathsf{C}$ then $(b,a)\in\mathsf{C}$. 
\item[(PC6)] Axiom of normality: If $(a,b)\notin\mathsf{C}$ then there exist
$c,d$ such that $(a,c)\notin\mathsf{C}$, $(d,b)\notin\mathsf{C}$,
and $c\vee d=1$. 
\item[(PC7)] Axiom of extensionality: \label{(PC7)}If $a\neq1$, then there exists
$b\neq0$ such that $(a,b)\notin\mathsf{C}$. 
\end{enumerate}
\begin{definition} \label{normal and compingent lattice} A precontact
lattice $\langle L,\mathsf{C}\rangle$ is called a contact lattice
if it satisfies axioms (PC4) and (PC5). A contact lattice $\langle L,\mathsf{C}\rangle$
is called normal if it satisfies axiom (PC6). A normal contact lattice
$\langle L,\mathsf{C}\rangle$ is called compingent if it satisfies
axiom (PC7). A contact lattice $\langle L,\mathsf{C}\rangle$ is called
extensional if it satisfies axiom (PC7). \end{definition}

\begin{remark} The notion of contact lattices was first introduced
in \cite{Duntsch1}. Normal contact lattices extend the class SubS5
algebras (see \cite{Bezha1}), called quasi-monadic algebras in \cite{Celani2}.
Extensional contact lattices extend extensional contact algebras studied,
for example, in \cite{Dimov2}. Moreover, when $L$ is a Boolean algebra,
condition (PC7) is equivalent to the following: 
\[
\mathsf{C}(a)\subseteq\mathsf{C}(b)\text{ implies }a\leq b,
\]
for all $a,b\in B$ (see, for instance \cite{Dimov2}). Finally, compingent
contact lattices extend compingent contact algebras studied by De
Vries algebras in \cite{DeVries}. \end{remark}

Most of the items in the following lemma are known results (see, for
example \cite{Duntsch1}). For the sake of completeness, we include
a proof of item (6).

\begin{lemma}\label{propp} Let $\langle L,\mathsf{C}\rangle$ be
a precontact lattice, and let $a,b\in L$. Then 
\begin{enumerate}
\item $\langle L,\mathsf{C}^{-1}\rangle$ is a precontact lattice. 
\item If $(a,b)\in\mathsf{C}$, $a\leq a'$ and $b\leq b'$ then $(a',b')\in\mathsf{C}$. 
\item If $a\neq0$ and $(a,a)\in\mathsf{C}$ then $(a,1)\in\mathsf{C}$. 
\item The sets $L\setminus\mathsf{C}(a)=\{b\in L:(a,b)\notin\mathsf{C}\}$
and $L\setminus\mathsf{C}^{-1}(a)=\{x\in L:(x,a)\notin\mathsf{C}\}$
are ideals of $L$. 
\item If $(a,b)\in\mathsf{C}$, then $[a)\times[b)\subseteq\mathsf{C}$. 
\item If $F$ and $G$ are filters of $L$ and $F\times G\subseteq\mathsf{C}$,
then there exist prime filters $P$ and $Q$ such that $F\subseteq P$,
$G\subseteq Q$, and $\textcolor{blue}{P\times Q \subseteq\mathsf{C}}$. 
\item For all $a,b\in L$, if $a\mathsf{C}b$, then there exist prime filters
$F$ and $G$ such that $F\times G\subseteq\mathsf{C}$, $a\in F$,
and $b\in G$. 
\end{enumerate}
\end{lemma} 
\begin{proof}
(6) Let $F$ and $G$ be filters such that $F\times G\subseteq\mathsf{C}$.
Consider the family 
\[
\mathcal{F}=\{P\in\operatorname{Fi}(L):F\subseteq P\text{ and }P\times G\subseteq\mathsf{C}\}.
\]
Clearly, $\mathcal{F}$ is nonempty since $F\in\mathcal{F}$. Moreover,
for every chain $\mathcal{C}\subseteq\mathcal{F}$, we have that $\bigcup\mathcal{C}\in\mathcal{F}$,
so by Zorn’s Lemma, $\mathcal{F}$ has a maximal element. Let $P$
be a maximal element of $\mathcal{F}$. We now show that $P$ is prime.
Suppose that $a,b\in L$ satisfy $a\vee b\in P$, but assume, for
contradiction, that $a,b\notin P$. Consider the filters $P_{a}=\mathrm{Fg}(P\cup\{a\})$
and $P_{b}=\mathrm{Fg}(P\cup\{b\})$. Then $P\subsetneq P_{a}$ and
$P\subsetneq P_{b}$, and hence $P\subsetneq P_{a}\cap P_{b}$. Since
$P$ is maximal in $\mathcal{F}$, we have $(P_{a}\cap P_{b})\times G\nsubseteq\mathsf{C}$.
Thus, there exist $x\in P_{a}\cap P_{b}$ and $y\in G$ such that
$(x,y)\notin\mathsf{C}$. Since $x\in P_{a}\cap P_{b}$, there exist
$p_{1},p_{2}\in P$ such that $p_{1}\wedge a\leq x$ and $p_{2}\wedge b\leq x$.
Let $p=p_{1}\wedge p_{2}$. Then, $p\wedge(a\vee b)\leq x$. Since
$p,a\vee b\in P$ and $P$ is upward closed, we conclude that $x\in P$.
This contradicts $(x,y)\notin\mathsf{C}$, since $P\times G\subseteq\mathsf{C}$.
Hence, $a$ or $b$ must be in $P$, i.e., $P$ is a prime filter.
Applying the same argument to $G$, we obtain a prime filter $Q$
such that $G\subseteq Q$ and $(P,Q)\in R_{\mathsf{C}}$. This completes
the proof. 
\end{proof}

Now we will present some examples of precontact relations in bounded
distributive lattices. 
\begin{example}
\label{trivial prec} In every bounded distributive lattice $L$,
one can define the overlap relation $\mathsf{O}$ by setting 
\[
(a,b)\in\mathsf{O}\Longleftrightarrow a\wedge b\neq0.
\]
This relation is a contact relation. Moreover, $\mathsf{O}$ is the
least contact relation on $L$ with respect to set inclusion. 
\end{example}

\begin{example}
\label{ex neg} The variety of bounded distributive lattices with
a negation operator, or $\neg$-lattices, was introduced in \cite{Celani6}.
A $\neg$-lattice is a pair $\langle L,\neg\rangle$ where $L$ is
a bounded distributive lattice and $\neg$ is a unary operator satisfying
the conditions, for all $a,b\in L$: 
\begin{enumerate}
\item $\neg0=1$, and 
\item $\neg(a\vee b)=\neg a\wedge\neg b$. 
\end{enumerate}
The operator $\neg$ is a generalization of well-known negations,
such as the intuitionistic negation or De Morgan negation.

In every $\neg$-lattice we can define a precontact relation $\mathsf{C}_{\neg}\subseteq L\times L$
as 
\[
(a,b)\in\mathsf{C}_{\neg}\Longleftrightarrow a\nleq\neg b\Longleftrightarrow a\in(\neg b]^{c}.
\]
Therefore, $\langle L,\mathsf{C}_{\neg}\rangle$ is a precontact lattice.
When $L$ is a Boolean algebra, the precontact relation $\mathsf{C}_{\neg}$
coincides with the overlap relation $\mathsf{O}$ defined in the Example
\ref{trivial prec}. We will refer to precontact lattices in which
the precontact relation is induced by a negation operator as negationally
definable precontact relation. We note that, in the setting of Boolean
algebras, the generalized negation operator is also referred to as
the sufficiency operator (see \cite{Duntsch-Ohrowska2001,Duntsch-Ohrowska-Tinchev2017})
\end{example}

\begin{example}
Let $L$ be a bounded distributive lattice and let $I\subseteq L$
be an ideal. We define two precontact relations on $L$ as follows: 
\begin{enumerate}
\item The relation $\mathsf{C}_{I}$ is given by $(a,b)\in\mathsf{C}_{I}\Longleftrightarrow a\wedge b\notin I$. 
\item The relation $\mathsf{C}^{I}$ is given by $(a,b)\in\mathsf{C}^{I}\Longleftrightarrow a,b\notin I$. 
\end{enumerate}
Both relations satisfy axiom (PC5). Moreover, since $I$ is an ideal,
it follows that if $a\wedge b\notin I$, then $a,b\notin I$. Therefore,
$\mathsf{C}_{I}\subseteq\mathsf{C}^{I}$. 
\end{example}

\begin{example}
\label{neg gen} We now consider an example that generalizes the notion
of negation introduced in Example \ref{ex neg}. Let $L$ be a bounded
distributive lattice. We say that a map $N:L\longrightarrow{\rm {Id}}(L)$
is a generalized negation if it satisfies the following conditions
for all $a,b\in L$: 
\begin{enumerate}
\item $N(0)=L$, 
\item $N(a\vee b)=N(a)\cap N(b)$. 
\end{enumerate}
We show that precontact relations in bounded distributive lattices
are in 1-1 correspondence with the generalized negation. Let $N$
be a generalized negation defined on a bounded distributive lattice
$L$. Then, it is easy to see that the relation $\mathsf{C}_{N}$
defined by 
\[
(a,b)\in\mathsf{C}_{N}\text{ iff }a\notin N(b),
\]
is a precontact relation on $L$. Conversely, if $\mathsf{C}$ is
a precontact relation defined on $L$, then the function $N_{\mathsf{C}}:L\longrightarrow{\rm {Id}}(L)$
defined by $N_{\mathsf{C}}(a)=(\mathsf{C}^{-1}(a))^{c}$ for all $a\in L$,
is a generalized negation defined on $L$. Moreover, $\mathsf{C}=\mathsf{C}_{N_{\mathsf{C}}}$
and $N=N_{\mathsf{C}_{N}}$. 
\end{example}

\begin{example}
\label{modal-prec-unified} Let $(L,\lozenge)$ be a $\lozenge$-modal
lattice. Define the relations $\mathsf{C}_{\lozenge}$ and $\mathsf{C}^{\lozenge}$
on $L$ by 
\[
\begin{array}{ccc}
(a,b)\in\mathsf{C}_{\lozenge} & \Longleftrightarrow & a\wedge\lozenge b\neq0,\\[1mm]
(a,b)\in\mathsf{C}^{\lozenge} & \Longleftrightarrow & \lozenge a\wedge\lozenge b\neq0.
\end{array}
\]
Then both $\mathsf{C}_{\lozenge}$ and $\mathsf{C}^{\lozenge}$ are
precontact relations on $L$. Moreover if $a\leq\lozenge a$, then
$\mathsf{C}_{\lozenge}\subseteq\mathsf{C}^{\lozenge}$. If $L$ is
a Boolean algebra the operator $\Box$ is such that $\lozenge x=\neg\Box\neg x$.
Then the relation $\mathsf{C}_{\lozenge}$ coincides with the relation
$\mathsf{C}_{\Box}$ given by 
\[
(a,b)\in\mathsf{C}_{\Box}\Longleftrightarrow a\nleq\Box\neg b.
\]
Following the terminology introduced in \cite{Bezha1}, we will refer
to Boolean precontact algebras in which the precontact relation arises
from a modal operator as modally definable precontact algebras. 
\end{example}

\begin{example}
\label{contacto-relacional} \cite{Dimov1,Duntsch2} Let $\langle X,R\rangle$
be a relational structure, where $X\neq\emptyset$ and $R\subseteq X\times X$
is a binary relation on $X$. We define a binary relation $\mathsf{C}_{R}$
on $\mathcal{P}(X)$ by 
\[
(U,V)\in\mathsf{C}_{R}\ \Longleftrightarrow\ (U\times V)\cap R\neq\emptyset,
\]
for all $U,V\subseteq X$. It is easy to verify that $\mathsf{C}_{R}$
is a precontact relation on the Boolean algebra $\mathcal{P}(X)$.
Moreover, we define the modal and negation operators: 
\[
\square_{R}(U)=\{x\in X:R(x)\subseteq U\}\quad\text{and}\quad\neg_{R}(U)=\{x\in X:R(x)\cap U=\emptyset\},
\]
for each $U\subseteq X$. Then, it is easy to see that: 
\[
(U,V)\in\mathsf{C}_{\square_{R}}\Longleftrightarrow U\nsubseteq\square_{R}(V^{c})\quad\text{and}\quad(U,V)\in\mathsf{C}_{\neg_{R}}\Longleftrightarrow U\nsubseteq\neg_{R}(V).
\]
It can be shown that the three relations coincide: 
\[
\mathsf{C}_{R}=\mathsf{C}_{\square_{R}}=\mathsf{C}_{\neg_{R}}.
\]
Thus, they provide equivalent ways of defining a precontact relation
on $\mathcal{P}(X)$. 
\end{example}

\begin{example}
This example is inspired by the results presented in \cite[Section 7]{Gruszczynski-Menchon2023}.
Let $L$ be a bounded distributive lattice. Let $S$ be a non-empty
subset of $L$. Define a relation $\mathsf{C}_{S}\subseteq L\times L$
as follows: 
\[
\left(a,b\right)\in\mathsf{C}_{S}\text{ iff }\exists c\in S\,\left(a\wedge c\neq0\text{ and }b\wedge c\neq0\right).
\]
Then, $\mathsf{C}_{S}$ is a symmetric precontact relation. Moreover,
when $\bigvee S=1,$ we can show that $\mathsf{C}_{S}${} satisfies
the axiom of reflexivity, so that $\left<L,\mathsf{C}_{S}\right>$
is a contact lattice in this case. 
\end{example}

In the literature on contact and precontact algebras, a commonly considered
notion of homomorphism is that of a Boolean homomorphism that reflects
the contact relation. Motivated by the connection between modal algebras
and precontact relations, we consider two classes of homomorphisms
precontact lattices.

\medskip{}
Let $\langle L_{1},\mathsf{C}_{1}\rangle$ and $\langle L_{2},\mathsf{C}_{2}\rangle$
be two precontact lattices. Let $h:{L}_{1}\longrightarrow{L}_{2}$
be a lattice homomorphism. We consider the following conditions: 
\begin{enumerate}
\item[$\mathrm{\mathrm{(CH1)}}$] $(h(a),h(b))\in\mathsf{C}_{2}$ implies that $(a,b)\in\mathsf{C}_{1}$,
for all $a,b\in L_{1}$. 
\item[$\mathrm{\mathrm{(CH2)}}$] If $(a,h(x))\notin\mathsf{C}_{2}$ then there exists $b\in L_{1}$
such that $(b,x)\notin\mathsf{C}_{1}$ and $a\leq h(b)$. 
\item[$\mathrm{\mathrm{(CH3)}}$] If $(h(x),a)\notin\mathsf{C}_{2}$ then there exists $b\in L_{1}$
such that $(x,b)\notin\mathsf{C}_{1}$ and $a\leq h(b)$. 
\end{enumerate}
\begin{remark}

In this context, condition $\mathrm{\mathrm{\mathrm{(CH3)}}}$ can
be seen as a dual formulation of $\mathrm{\mathrm{\mathrm{(CH2)}}}$.
Moreover, when $\mathsf{C}_{{1}}$ and $\mathsf{C}_{2}$ are symmetric,
the conditions $\mathrm{\mathrm{\mathrm{(CH2)}}}$ and $\mathrm{\mathrm{\mathrm{(CH3)}}}$
are equivalent.

\end{remark}

\begin{definition}\label{hom_prec} Let $h:\langle L_{1},\mathsf{C}_{1}\rangle\rightarrow\langle L_{2},\mathsf{C}_{2}\rangle$
be a map between precontact structures. We say that $h$ is a 
\begin{itemize}
\item precontact homomorphism if it is a lattice homomorphism that satisfies
the condition $\mathrm{\mathrm{\mathrm{(CH1)}}}$. 
\item strong precontact homomorphism if it is a lattice homomorphism that
satisfies both $\mathrm{\mathrm{\mathrm{(CH1)}}}$ and $\mathrm{\mathrm{\mathrm{(CH2)}}}$. 
\item precontact isomorphism if $h$ is a lattice isomorphism such that
$(a,b)\in\mathsf{C}_{1}$ if and only if $(h(a),h(b))\in\mathsf{C}_{2}$
for all $a,b\in L_{1}$. 
\item strong precontact isomorphism if $h$ is a lattice isomorphism and
satisfies the conditions of a strong precontact homomorphism. 
\end{itemize}
\end{definition}
\begin{example}
Let $\langle L,\mathsf{C}\rangle$ be a precontact lattice. The identity
map $\mathrm{id}_{L}:L\to L$ is a strong precontact isomorphism. 
\end{example}

\begin{example}
Let $L_{1}=\{0,a,b,1\}$ the Boolean lattice with atoms $a$ and $b$,
and let $L_{2}=\{0,1\}$. Define a lattice homomorphism $h:L_{1}\to L_{2}$
by 
\[
h(0)=0,\,h(a)=0,\,h(b)=0,\,h(1)=1.
\]
Let $\mathsf{C}_{1}=\{(1,1),(1,b),(1,a),(a,1),(a,b),(b,1)\}$ and
let $\mathsf{C}_{2}=\{(1,1)\}$. Then $h:(L_{1},\mathsf{C}_{1})\to(L_{2},\mathsf{C}_{2})$
is a precontact homomorphism, since it satisfies (CH1). However, $h$
does not satisfy (CH2) because $(1,h(a))\notin\mathsf{C}_{2}$ and
for all $t\in L_{1}$ such that $1\leq h(t)$ we obtain that $(t,a)\in\mathsf{C}_{1}$.
Moreover, one can also check that condition (CH3) fails. 
\end{example}

\begin{remark} By the previous example, we conclude that the notion
of precontact homomorphism does not coincide with that of strong precontact
homomorphism. In particular, a precontact homomorphism need not be
a strong precontact homomorphism. \end{remark} In this work, we will
focus on precontact homomorphisms and strong precontact homomorphisms.
Additionally, homomorphisms that satisfy (CH1) and (CH3) can be studied
in a manner dual to strong precontact homomorphisms. That is, while
strong precontact homomorphisms require the additional condition $\mathrm{\mathrm{\mathrm{(CH2)}}}$,
their dual counterpart would consider the condition $\mathrm{(CH3)}$.

Let $(A,\textcolor{blue}{\square_{1}})$ and $(B,\textcolor{blue}{\square_{2}})$
be modal algebras. Recall that a modal homomorphism $h:(A,\square_{1})\rightarrow(B,\square_{2})$
is a Boolean homomorphism such that $h(\square_{1}a)=\square_{2}(h(a))$
for all $a\in A$. We show that, in the Boolean modal setting, strong
precontact homomorphisms coincide with the usual notion of modal homomorphism.
The following result requires recalling the precontact relation defined
via a modal operator, as introduced in Example \ref{modal-prec-unified}.

\begin{proposition}\label{homo modales} \label{relacion con homo booleano}
Let $(A,\square_{1})$ and $(B,\square_{2})$ be modal algebras, and
let $h:A\rightarrow B$ be a Boolean homomorphism. Then: 
\begin{enumerate}
\item $h(\square_{1}a)\leq\square_{2}h(a)$ for all $a\in A$ if and only
if $h:(A,\mathsf{C}_{\square_{1}})\rightarrow(B,\mathsf{C}_{\square_{2}})$
satisfies condition (CH1). 
\item $\square_{2}h(a)\leq h(\square_{1}a)$ for all $a\in A$ if and only
if $h(A,\mathsf{C}_{\square_{1}})\rightarrow(B,\mathsf{C}_{\square_{2}})$
satisfies condition (CH2). 
\item $h(\square_{1}a)=\square_{2}h(a)$ for all $a\in A$ if and only if
$h:(A,\mathsf{C}_{\square_{1}})\rightarrow(B,\mathsf{C}_{\square_{2}})$
is a strong precontact homomorphism. 
\end{enumerate}
\end{proposition} 
\begin{proof}
$(1)$ Assume that $h(\square_{1}a)\leq\square_{2}h(a)$ for all $a\in A$.
Suppose that there $a,b\in A$ such that $(a,b)\notin\mathsf{C}_{\square_{1}}$.
By the hypothesis and using that $h$ is a Boolean homomorphism, we
have $h(a)\leq h(\square_{1}\neg b)\leq\square_{2}h(\neg b)=\square_{2}\neg h(b)$,
so $(h(a),h(b))\notin\mathsf{C}_{\square_{2}}$. Conversely, suppose
that $h:(A,\mathsf{C}_{\square_{1}})\rightarrow(B,\mathsf{C}_{\square_{2}})$
satisfies condition (CH1). Let $a\in A$. Since $(\square_{1}a,\neg a)\notin\mathsf{C}_{\square_{1}}$,
it follows that $(h(\square_{1}a),h(\neg a))\notin\mathsf{C}_{\square_{2}}$.
This implies $h(\square_{1}a)\leq\square_{2}\neg(h(\neg a))=\square_{2}h(a)$.

$(2)$ Assume that $\square_{2}h(a)\leq h(\square_{1}a)$ for all
$a\in A$. Let $(x,h(b))\notin\mathsf{C}_{\square_{2}}$, so $x\leq\square_{2}\neg h(b)$.
By hypothesis and since $h$ is a Boolean homomorphism, $\square_{2}\neg h(b)=\square_{2}h(\neg b)\leq h(\square_{1}\neg b)$.
Taking $c=\square_{1}\neg b$, we conclude that there exists $c\in A$
such that $(c,b)\notin\mathsf{C}_{\square_{1}}$ and $x\leq h(c)$.

Conversely, suppose $h:(A,\mathsf{C}_{\square_{1}})\longrightarrow(B,\mathsf{C}_{\square_{2}})$
satisfies condition {(CH2)}. Let $a\in A$. Since $(\square_{2}h(a),h(\neg a))\notin\mathsf{C}_{\square_{2}}$,
there exists $b\in A$ such that $(b,\neg a)\notin\mathsf{C}_{\square_{1}}$
and $\square_{2}h(a)\leq h(b)$. So, $b\leq\square_{1}a$. As $h$
is a Boolean homomorphism, it is in particular a lattice homomorphism
and therefore order-preserving; consequently, $h(b)\leq h(\square_{1}a)$.
Thus, $\square_{2}h(a)\leq h(\square_{1}a)$. 

$(3)$ Follows immediately from $(1)$ and $(2)$. 
\end{proof}

It is easy to see that the composition of precontact homomorphisms
(strong precontact homomorphisms) is a precontact homomorphism (strong
precontact homomorphism)

Now we are in a position to define two algebraic categories. Both
categories have the same objects but differ in their morphisms.
\begin{center}
\begin{tabular}{|c|c|}
\hline 
Category  & Details \tabularnewline
\hline 
PreDL  & Objects: Precontact Distributive Lattices \tabularnewline
 & Morphisms: Precontact Homomorphisms \tabularnewline
\hline 
PreDLS  & Objects: Precontact Distributive Lattices \tabularnewline
 & Morphisms: Strong Precontact Homomorphisms \tabularnewline
\hline 
\end{tabular}
\par\end{center}

\section{Representation and duality}

\label{sec:Priestley-duality}

This section is devoted to the study of the representation and Priestley
duality theory for precontact lattices.

Let $\langle L,\mathsf{C}\rangle$ be a precontact lattice. To simplify
the notation and to be able to make a comparison to some results about
lattice with negation from \cite{Celani6} (especially the representation
theory), we introduce the maps 
\[
\alpha,\lambda:L\to\mathrm{Id}(L)
\]
defined by 
\[
\alpha(a)=L\setminus\mathsf{C}^{-1}(a)\quad\text{and}\quad\lambda(a)=L\setminus\mathsf{C}(a),
\]
for each $a\in L$. Both $\alpha$ and $\lambda$ are generalized
negations (see Example~\ref{neg gen}). The fact that these are well-defined
as maps into $\mathrm{Id}(L)$ follows from item~(4) in Lemma~\ref{propp},
which ensures that, for each $a\in L$, the sets $\alpha(a)$ and
$\lambda(a)$ are ideals of $L$. These will be used to define the
following subsets. For each $B\subseteq L$ let: 
\begin{center}
$\alpha^{-1}(B)=\{a:\alpha(a)\cap B\neq\emptyset\}$ 
\par\end{center}

and 
\begin{center}
$\lambda^{-1}(B)=\{a:\lambda(a)\cap B\neq\emptyset\}$. 
\par\end{center}

To provide a topological representation for distributive precontact
lattices we define a binary relation $R_{\mathsf{C}}$ on $\mathrm{X}(L)$
as follows: 
\[
(P,Q)\in R_{\mathsf{C}}\text{ iff }P\times Q\subseteq\mathsf{C}.
\]

We observe that the definition of the relation $R_{\mathsf{C}}$ is
a generalization of the similar relation on ultrafilters defined in
the framework of Boolean contact algebras (\cite{Dimov1,Dimov3,Duntsch2}).
This relation is referred to as the canonical relation of $\mathsf{C}$,
and the relational system $\langle{\rm {X}}(L),R_{\mathsf{C}}\rangle$
is called the canonical structure of $\langle L,\mathsf{C}\rangle$
or the discrete precontact space of $\langle L,\mathsf{C}\rangle$.

We now present several equivalent formulations of the relation $R_{\mathsf{C}}$,
which plays a central role throughout the paper. In particular, the
characterization involving the set $\alpha^{-1}(P)$ will be used
as our main working definition. A dual version, based on $\lambda^{-1}(Q)$,
will also be useful in symmetric arguments.

\begin{lemma}\label{R alternativas} Let $\langle L,\mathsf{C}\rangle$
be a precontact lattice, and let $P,Q\in{\rm X}(L)$. The following
conditions are equivalent: 
\begin{enumerate}
\item $P\times Q\subseteq\mathsf{C}$. 
\item $\alpha^{-1}(P)\cap Q=\emptyset$. 
\item $\lambda^{-1}(Q)\cap P=\emptyset$. 
\end{enumerate}
\end{lemma} 
\begin{proof}
We will only prove $(1)\Leftrightarrow(2)$, since the equivalence
$(2)\Leftrightarrow(3)$ is straightforward (see \cite{Duntsch1}).
Thus, reasoning by contraposition: 
\[
\begin{array}{rl}
\alpha^{-1}(P)\cap Q\neq\emptyset & \Leftrightarrow\ \exists b\in Q,\ \exists a\in P\text{ such that }a\in\alpha(b)\\
 & \Leftrightarrow\ \exists(a,b)\in P\times Q\text{ such that }(a,b)\notin\mathsf{C}\\
 & \Leftrightarrow\ P\times Q\nsubseteq\mathsf{C}.
\end{array}
\]
\end{proof}

The following lemmas play a fundamental role in the development of
the theory of precontact lattices and will be used repeatedly throughout
the paper.

\begin{lemma}\label{separacion} Let $\langle L,\mathsf{C}\rangle$
be a precontact lattice. Then for each $P\in{\rm X}(L)$, the set
$\alpha^{-1}(P)$ is an ideal. Moreover $a\notin\alpha^{-1}(P)$ iff
there exists $Q\in{\rm X}(L)$ such that $\alpha^{-1}(P)\cap Q=\emptyset$
and $a\in Q$.

\end{lemma} 
\begin{proof}
It is easy to verify that $\alpha^{-1}(P)$ is an ideal. Now let $a\notin\alpha^{-1}(P)$.
Then $\mathrm{Fg}(a)\cap\alpha^{-1}(P)=\emptyset$. Since $\alpha^{-1}(P)$
is an ideal, the Birkhoff--Stone theorem guarantees the existence
of a prime filter $Q$ such that $\alpha^{-1}(P)\cap Q=\emptyset$
and $a\in Q$. The converse direction is immediate. 
\end{proof}

We now present the dual version of the previous result in terms of
$\lambda$, which will be useful in symmetric constructions.

\begin{lemma} Let $\langle L,\mathsf{C}\rangle$ be a precontact
lattice. Then for each $P\in{\rm X}(L)$, the set $\lambda^{-1}(P)$
is an ideal. Moreover $a\notin\lambda^{-1}(P)$ iff there exists $Q\in{\rm X}(L)$
such that $\lambda^{-1}(Q)\cap P=\emptyset$ and $a\in Q$. \end{lemma}

In the context of precontact lattices and their associated relational
structures, closed relations play a central role in the representation
theory. Let us recall how they are defined.

\begin{definition} Let $X$ be a topological space and let $R$ be
a binary relation on $X$. We recall that $R$ a closed relation provided
$R$ is a closed set in the product topology on $X\times X$. We shall
say that $R$ is a point-closed relation, if $R(x)$ is closed subset
of $X$, for each $x\in X$. \end{definition}

In order to define the dual spaces of precontact lattices, we begin
by establishing a characterization of the kind of relations involved
in Priestley spaces. This result will serve as the basis for the definition
that follows.

\begin{lemma}{\label{Caracterizacion de espacio}} Let $\langle X,\leq,\tau\rangle$
be a Priestley space and let $R$ be a binary relation defined on
$X$. The following conditions are equivalent: 
\begin{itemize}
\item[(1)] $R(x)$ is a closed downset of $X$ for each $x\in X$, and $R^{-1}(U)^{c}$
is an open upset, for each $U\in D(X)$. 
\item[(2)] 
\begin{enumerate}
\item[(i)] If $Y\subseteq X$ is closed subset of $X$, then $R(Y)$ is closed
and downset. 
\item[(ii)] If $Z$ is closed downset, then $R^{-1}(Z)$ is closed downset. 
\end{enumerate}
\end{itemize}
\end{lemma} 
\begin{proof}
$(1)\Rightarrow(2)$ $(i)$ Let $Y\subseteq X$ be closed. Suppose
that $x\notin R(Y)$, i.e, $x\notin R(y)$ for all $y\in Y$. As $R(y)$
is closed downset for all $y\in Y$, for each $y\in Y$ there exists
$U_{y}\in{\rm {D}(X)}$ such that $x\in U_{y}$ and $R(y)\cap U_{y}=\emptyset$.
That is, for each $y\in Y$ there exists $U_{y}\in{\rm {D}}(X)$ such
that $x\in U_{y}$ and $y\in R^{-1}(U_{y})^{c}$. So $Y\subseteq\bigcup\{R^{-1}(U_{y})^{c}:y\in Y\}$.
As $Y$ is closed and $X$ is compact, $Y$ is compact. Thus, there
exist $y_{1},y_{2},\ldots,y_{n}\in Y$ such that 
\[
Y\subseteq R^{-1}(U_{y_{1}})^{c}\cup\ldots\cup R^{-1}(U_{y_{n}})^{c}\subseteq R^{-1}(U_{y_{1}}\cup\ldots\cup Uy_{n})^{c}
\]
Let $U_{x}=U_{y_{1}}\cup\ldots\cup U_{y_{n}}$. Then, $R(Y)\subseteq U^{c}_{x}$
and $x\in U_{x}$. So, $R(Y)=\bigcap\{U^{c}_{x}:x\notin R(Y)\}$.
Therefore, $R(Y)$ is closed and downset.

(ii) Let $Z$ be closed and upset. Let $x\notin R^{-1}(Z)$, i.e.
$Z\cap R(x)=\emptyset$. As $Z$ is a closed and upset, $Z=\bigcap\{U_{i}:Z\subseteq U_{i},U_{i}\in{\rm {D}(X)}\}$.
Therefore, $R(x)\subseteq\bigcup\{U^{c}_{i}:Z\subseteq U_{i},U_{i}\in{\rm {D}(X)}\}$.
As $R(x)$ is closed, there exists a finite subset $\{U_{1},\ldots,U_{n}\}$
such that $R(x)\subseteq U^{c}_{1}\cup\ldots\cup U^{c}_{n}$. Let
$U_{x}=U_{1}\cap\ldots\cap U_{n}$. Then, $R(x)\cap U_{x}=\emptyset$.
So, $x\notin R^{-1}(U_{x})$. As $Z\subseteq U_{i}$ for all $i=1,\ldots,n$,
then $Z\subseteq U_{x}$. Thus, $R^{-1}(Z)\subseteq R^{-1}(U_{x})$
and $x\notin R^{-1}(U_{x})$. In consequence, $R^{-1}(Z)=\bigcap\{R^{-1}(U_{x}):x\notin R^{-1}(Z)\}$
. Hence $R^{-1}(Z)$ is a closed downset.

$(2)\Rightarrow(1)$ As $X$ is Hausdorff, $\{x\}$ is closed. Thus,
$R(\{x\})=R(x)$ is a closed downset. By property $(ii)$, $R^{-1}(U)$
is a closed downset, therefore $R^{-1}(U)^{c}$ is an open upset. 
\end{proof}

We are now in a position to define the dual spaces of precontact lattices.

\begin{definition}\label{espacioprecontacto} A Priestley precontact
space is a pair $\langle\mathbb{X},R\rangle$ such that $\mathbb{X}=\langle X,\leq,\tau\rangle$
is a Priestley space satisfying any of the conditions (1) or (2) of
Lemma \ref{Caracterizacion de espacio}.

\begin{remark} \label{negation spaces} 
\begin{enumerate}
\item Let $\langle X,R\rangle$ be a relational structure. We define 
\[
\neg_{R}(U)\coloneqq R^{-1}(U)^{c},
\]
for each $U\in\mathsf{D}(X)$. Then $\neg_{R}(\emptyset)=X$ and $\neg_{R}(U\cup V)=\neg_{R}(U)\cap\neg_{R}(V)$
for all $U,V\in\mathsf{D}(X)$. When $\neg_{R}(U)\in\mathsf{D}(X)$
for all $U\in\mathsf{D}(X)$, the precontact spaces coincide with
the $\neg$-spaces defined in \cite{Celani6}. 
\item If $\langle X,R\rangle$ is a Priestley precontact space, then $\left(\leq\circ R\right)\subseteq R$.
This follows immediately from the fact that $R^{-1}(z)$ is a downset
for every $z\in X$. Indeed, if $x\leq y$ and $(y,z)\in R$, then
$x\in R^{-1}(z)$, so $(x,z)\in R$. 
\end{enumerate}
\end{remark}

\begin{proposition}\label{precontacto canonico} Let $\mathcal{F}=\langle X,R\rangle$
be a precontact space. Let $\mathsf{C}_{R}$ be a binary relation
on ${D}(X)$ given by: 
\[
\left(U,V\right)\in\mathsf{C}_{R}\text{ iff }\left(U\times V\right)\cap R\neq\emptyset.
\]
Then, 
\[
\mathcal{A}(\mathcal{F})=\langle{\rm {D}(X)},\cup,\cap,\mathsf{C}_{R},\emptyset,X\rangle
\]
is a precontact lattice. \end{proposition} 
\begin{proof}
It is left to the reader to prove that $\mathsf{C}_{R}$ is a precontact
relation on the lattice ${\rm {D}(X)}$ (see Example \ref{contacto-relacional}). 
\end{proof}

\end{definition}

By Proposition \ref{precontacto canonico}, we conclude that every
precontact space induces a precontact lattice. We now show that every
precontact lattice gives rise to a Priestley precontact space.

\begin{proposition} Let $\langle L,C\rangle$ be a precontact lattice.
Then 
\[
\mathcal{F}(L)=\langle X(L),\subseteq,R_{C},\tau_{X(L)}\rangle
\]
is a Priestley precontact space. \end{proposition}
\begin{proof}
First, let us prove that for all $P\in{\rm {X}(L)},R_{\mathsf{C}}(P)$
is closed and downset. To prove this, let us see that for all $P\in\mathsf{X}(L)$
\[
R_{\mathsf{C}}(P)=\bigcap\{\beta(a)^{c}:P\nsubseteq\mathsf{C}^{-1}(a)\}.
\]
Let $Q\in R_{\mathsf{C}}(P)$. Suppose that $Q\notin\bigcap\{\beta(a)^{c}:P\nsubseteq\mathsf{C}^{-1}(a)\}$.
Then there exists $a\in L$ such that $a\in Q$ and $P\nsubseteq\mathsf{C}^{-1}(a)$;
that is, there exists $p\in P$ such that $(p,a)\notin\mathsf{C}$.
This is a contradiction, since $(p,a)\in P\times Q$ and $P\times Q\subseteq\mathsf{C}$.

Assume that $Q\in\bigcap\{\beta(a)^{c}:P\nsubseteq\mathsf{C}^{-1}(a)\}$.
If $Q\notin R_{\mathsf{C}}(P)$, there exist $p,q\in L$ such that
$\left(p,q\right)\in P\times Q$ but $\left(p,q\right)\notin\mathsf{C}$.
Then, $p\in P$ and $p\notin\mathsf{C}^{-1}(q)$, i.e., $P\nsubseteq\mathsf{C}^{-1}(q)$.
So, $Q\in\beta(q)^{c},$ which is a contradiction. Thus, $Q\in R_{\mathsf{C}}(P).$

Now we prove that $R^{-1}_{\mathsf{C}}(U)^{c}$ is an open upset of
$\mathrm{X}(L)$, for each $U\in\mathrm{D}(\mathrm{X}(L))$, the set.
To this end, it suffices to show that for each $a\in L$ 
\[
\varphi(L\setminus\mathsf{C}^{-1}(a))=R^{-1}_{\mathsf{C}}(\beta(a))^{c}.
\]
Let $P\in\varphi(L\setminus\mathsf{C}^{-1}(a))$, i.e. , there exists
$p\in P$ such $\left(p,a\right)\notin\mathsf{C}$. If $P\notin R^{-1}_{\mathsf{C}}(\beta(a))^{c}$,
there exists $Q\in\mathrm{X}(L)$ such that $P\times Q\subseteq\mathsf{C}$
and $a\in Q$. This implies that $\left(p,a\right)\in P\times Q\subseteq\mathsf{C}$,
which is a contradiction. For the reverse inclusion let $P\in R^{-1}_{\mathsf{C}}(\beta(a))^{c}$.
Then $R_{\mathsf{C}}(P)\cap\beta(a)=\emptyset.$ If $P\cap\left(L\setminus\mathsf{C}^{-1}(a)\right)=\emptyset$.
By Lemma \ref{separacion}, there exists $Q\in\mathrm{X}(L)$ such
that $P\times Q\subseteq\mathsf{C}$ and $a\in Q$. So, $Q\in R_{\mathsf{C}}(P)\cap\beta(a)$,
which is a contradiction. Therefore, $\varphi(L\setminus\mathsf{C}^{-1}(a))=R^{-1}_{\mathsf{C}}(\beta(a))^{c}$. 
\end{proof}

In line with the notions of precontact homomorphism and strong precontact
homomorphism, we will define two types of morphisms between Priestley
precontact spaces.

Let $f:\langle{X}_{1},R_{1}\rangle\longrightarrow\langle{X}_{2},R_{2}\rangle$
be a Priestley morphism. We consider the following conditions: 
\begin{itemize}
\item[$\mathrm{(SP1)}$] If $(x,y)\in R_{1}$, then $(f(x),f(y))\in R_{2}$. 
\item[$\mathrm{\mathrm{(SP2)}}$] If $(f(x),y)\in R_{2}$, then there exists $z\in X_{1}$ such that
$(x,z)\in R_{1}$ and $y\leq f(z)$. 
\item[$\mathrm{(SP3)}$] If $(x,f(y))\in R_{2}$, then there exists $z\in X_{1}$ such that
$(z,y)\in R_{1}$ and $x\leq f(z)$. 
\end{itemize}
\begin{definition} We say that a Priestley morphism $f:\langle{X}_{1},R_{1}\rangle\longrightarrow\langle{X}_{2},R_{2}\rangle$
is stable if it satisfies $\mathrm{(SP1)}$. We say that $f$ is a
strong stable morphism if it satisfies $\mathrm{(SP1)}$ and $\mathrm{(SP2)}$.
\end{definition}

The condition $\mathrm{(SP1)}$ ensures that the relation $R_{1}$
is preserved under the morphism $f$. Conditions $\mathrm{(SP2)}$
and $\mathrm{(SP3)}$ are dual to each other, reflecting complementary
aspects of stability. This duality mirrors the behavior of conditions
$\mathrm{\mathrm{(CH2)}}$ and $\mathrm{\mathrm{(CH3)}}$ in precontact
homomorphisms.

It is easy to see that the composition of stable morphisms (or strong
stable morphisms) is a stable morphism (strong stable morphism). So,
we consider two categories:
\begin{center}
\begin{tabular}{|c|c|}
\hline 
Category  & Details \tabularnewline
\hline 
PS  & Objects: Precontact Spaces \tabularnewline
 & Morphisms: Stable Priestley Morphisms \tabularnewline
\hline 
SPS  & Objects: Precontact Spaces \tabularnewline
 & Morphisms: Strong Stable Morphisms \tabularnewline
\hline 
\end{tabular}
\par\end{center}

%%%%%%%%%%%%%%%%%%%%%%%%%%%%%%%%%%%

The celebrated Priestley duality states that the category of bounded
distributive lattices and lattice homomorphisms is dually equivalent
to the category of Priestley spaces and continuous order-preserving
maps. We extend the Priestley duality to the categories PS and SPS.

\begin{lemma} Let $\langle{L}_{1},\mathsf{C}_{1}\rangle,\langle{L}_{2},\mathsf{C}_{2}\rangle$
be a precontact lattice and let $h:{L}_{1}\longrightarrow{L}_{2}$
be a lattice homomorphism. Consider the map 
\[
h_{*}:{\rm {X}}(L_{2})\longrightarrow{\rm {X}}(L_{1})
\]
defined by $h_{*}(P)=h^{-1}(P)$, for each $P\in{\rm {X}}(L_{2})$.
Then $h_{*}$ is a Priestley morphism, i.e. continuos and order preserving,
such that: 
\begin{enumerate}
\item $h$ is a precontact homomorphism iff $h_{*}$ is a stable morphism. 
\item $h$ is a strong precontact homomorphism iff $h_{*}$ is a strong
stable morphism. 
\end{enumerate}
\end{lemma} 
\begin{proof}
$(1)$ Assume that $h$ satisfies condition $\mathrm{\mathrm{(CH1)}}$.
Let $(P,Q)\in R_{2}$. We aim to prove that $(h_{*}(P),h_{*}(Q))\in R_{1}$,
i.e. , $h^{-1}(P)\times h^{-1}(Q)\subseteq C_{1}$. Consider $(p,q)\in h^{-1}(P)\times h^{-1}(Q)$.
Then, $(h(p),h(q))\in P\times Q$. Since $P\times Q\subseteq\mathsf{C}_{2}$
by hypothesis, it follows that $(h(p),h(q))\in C_{2}$. By the stability
of $h$, we conclude that $(p,q)\in C_{2}$.

Assume that $h_{*}$ satisfies $\mathrm{(SP1)}$. Let $(h(a),h(b))\in\mathsf{C}_{2}$.
Then, by $(7)$ of Lemma \ref{propp}, there exists $P,Q\in{\rm {X}}(L_{2})$
such that $h(a)\in P,h(b)\in Q$ and $(P,Q)\in R_{2}$. As $(P,Q)\in R_{2}$
and $h_{*}$ is stable, $(h^{-1}(P),h^{-1}(Q))\in R_{1}$. Thus, $(a,b)\in h^{-1}(P)\times h^{-1}(Q)\subseteq C_{1}$.
So, $(a,b)\in C_{1}$.

$(2)$ We assume that $h\colon\langle L_{1},\mathsf{C}_{1}\rangle\longrightarrow\langle L_{2},\mathsf{C}_{2}\rangle$
satisfies conditions $\mathrm{\mathrm{(CH1)}}$ and $\mathrm{\mathrm{\mathrm{\mathrm{(CH2)}}}}$.
By $(1)$, $h_{*}$ satisfies $\mathrm{(SP1)}$. Therefore, it only
remains to prove that $h_{*}$ satisfies $\mathrm{(SP2)}$. Consider
$P\in\mathrm{X}(L_{2})$ and $Q\in\mathrm{X}(L_{1})$ such that $(h_{*}(P),Q)\in R_{1}$,
that is, $h^{-1}(P)\times Q\subseteq\mathsf{C}_{1}$.

We prove that $P\times\mathrm{Fg}(h(Q))\subseteq\mathsf{C}_{2}.$
Assume that there exist $p\in P$ and $x\in Q$ such that $(p,h(x))\notin\mathsf{C}_{2}$.
Then, by condition $\mathrm{(CH2)}$, there exists $b\in L_{1}$ such
that $(b,x)\notin\mathsf{C}_{1}$ and $p\leq h(b)$. However, this
leads to a contradiction, because $p\leq h(b)$ and $p\in P$ imply
that $b\in h^{-1}(P)$. Thus, $(b,x)\in h^{-1}(P)\times Q\subseteq\mathsf{C}_{1}$,
i.e. , $(b,x)\in\mathsf{C}_{1}$, which is a contradiction. Thus,
$P\times\mathrm{Fg}(h(Q))\subseteq\mathsf{C}_{2}.$ By item $(6)$
of Lemma \ref{propp}, there exist $D,F\in\mathrm{X}(L_{2})$ such
that $\mathrm{Fg}(h(Q))\subseteq D$, $P\subseteq F$, and $(F,D)\in R_{2}$.
This implies that $(P,D)\in R_{2}$. Therefore, there exists $D\in\mathrm{X}(L_{2})$
such that $(P,D)\in R_{2}$ and $Q\subseteq h^{-1}(D)$.

Now let us prove that if $h_{*}$ is a strong stable morphism, then
$h$ is a strong precontact homomorphism. By the previous theorem,
and since $h_{*}$ is a stable map, $h$ is a precontact homomorphism.
Therefore, it remains to show that $h$ satisfies $\mathrm{(CH2)}$.
Assume that $(h_{*}(P),Q)\in R_{1}$ implies that there exists $D\in{\mathrm{X}}(L_{2})$
such that $(P,D)\in R_{2}$ and $Q\subseteq h^{-1}(D)$. Let $(a,h(x))\notin\mathsf{C}_{2}$.
Suppose that for all $b\in L_{1}$, if $a\leq h(b)$ then $(b,x)\in\mathsf{C}_{1}$.
Hence, $h^{-1}(\mathrm{Fg}(a))\times\mathrm{Fg}(x)\subseteq\mathsf{C}_{1}$.
Thus, by $(6)$ of Lemma \ref{propp}, there exist $F,Q\in{\mathrm{X}}(L_{1})$
such that $h^{-1}(\mathrm{Fg}(a))\subseteq F$, $[x)\subseteq Q$,
and $(F,Q)\in R_{1}$. Since $h^{-1}(\mathrm{Fg}(a))\subseteq F$,
there exists $P\in{\mathrm{X}}(L_{1})$ such that $h^{-1}(P)\subseteq F$
and $\mathrm{Fg}(a)\subseteq P$. As $h^{-1}(P)\subseteq F$ and $(F,Q)\in R_{1}$,
it follows that $(h^{-1}(P),Q)\in R_{1}$. Then, applying the hypothesis,
there exists $D\in{\rm {X}}(L_{2})$ such that $(P,D)\in R_{2}$ and
$Q\subseteq h^{-1}(D)$. Consequently, $(a,h(x))\in\mathsf{C}_{2}$,
which is a contradiction. Therefore, there exists $b\in L_{1}$ such
that $(b,x)\notin\mathsf{C}_{1}$ and $a\leq h(b)$. 
\end{proof}

Based on the following result, every distributive precontact lattice
can be represented as a distributive precontact lattice of sets. This
theorem extends the representation result established for precontact
algebras in \cite{Duntsch2}.

\begin{theorem}[Representation] Let $\langle L,\mathsf{C}\rangle$
be a precontact lattice. Then the map $\beta:\langle L,\mathsf{C}\rangle\longrightarrow\langle\beta(L),\mathsf{C}_{R_{\mathsf{C}}}\rangle$
is a precontact isomorphism, i.e., $\beta$ is a lattice isomorphism
such that 
\[
(a,b)\in\mathsf{C}\text{ iff }\left(\beta(a),\beta(b)\right)\in\mathsf{C}_{R_{\mathsf{C}}},
\]
for all $a,b\in L$. \end{theorem} 
\begin{proof}
By the representation theorem for bounded distributive lattices, we
know that $\beta:L\to\beta(L)$ is a lattice isomorphism. Therefore,
it suffices to show that $(a,b)\in\mathsf{C}$ if and only if $\beta(a)\,\mathsf{C}_{R_{\mathsf{C}}}\,\beta(b)$,
for all $a,b\in L$.

Let $a,b\in L$ be such that $(a,b)\in\mathsf{C}$. By item (7) of
Lemma \ref{propp}, there are $P,Q\in\mathrm{X}(L)$ such that $a\in P$,
$b\in Q$, and $(P,Q)\in R_{\mathsf{C}}$. Consequently, $(\beta(a)\times\beta(b))\cap R_{\mathsf{C}}\neq\varnothing$,
which means that $(\beta(a),\beta(b))\in\mathsf{C}_{R_{\mathsf{C}}}$.

Conversely, suppose that $(\beta(a),\beta(b))\in\mathsf{C}_{R_{\mathsf{C}}}$.
Then $(\beta(a)\times\beta(b))\cap R_{\mathsf{C}}\neq\varnothing$,
so there exist $P,Q\in\mathrm{X}(L)$ such that $a\in P$, $b\in Q$,
and $(P,Q)\in R_{\mathsf{C}}$. By the definition of $R_{\mathsf{C}}$
and by the fact that $(a,b)\in P\times Q$, we have $(a,b)\in\mathsf{C}$. 
\end{proof}

Let $\langle X,\leq,\tau\rangle$ be a Priestley space. Consider the
map $\varepsilon_{X}:X\longrightarrow X({\rm {D}(X)})$ defined by
$\varepsilon_{X}(x)=\{U\in{\rm {D}(X)}:x\in U\}$. It follows from
Priestley duality that $\varepsilon_{X}$ is an order-isomorphism
and a homeomorphism.

\begin{lemma} \label{caract R}

Let $\langle X,R\rangle$ be a pair such that $X$ is a Priestley
space, $R$ is a binary relation defined on $X$ and $R^{-1}(U)^{c}$
is open upset of $X$, for each $U\in{\rm {D}}(X)$. Then the following
conditions are equivalent: 
\begin{enumerate}
\item For all $x\in X:R(x)$ is a closed downset. 
\item For all $x,y\in X$, if $(\varepsilon(x),\varepsilon(y))\in R_{\mathsf{C}_{R}}$,
then $(x,y)\in R$. 
\end{enumerate}
\end{lemma} 
\begin{proof}
$(1)\Rightarrow(2)$ Assume that for all $x\in X$, $R(x)$ is closed
downset. Let $x,y\in X$. Suppose that $(x,y)\notin R$, then since
$R$ is closed downset, there exists $V\in\mathsf{D}(X)$ such that
$y\in V$ and $R(x)\cap V=\emptyset$. In other words, there exists
$V\in\mathsf{D}(X)$ such that $y\in V$ and $x\in R^{-1}(V)^{c}$.
Since $x\in R^{-1}(V)^{c}$ and $R^{-1}(V)^{c}$ is open, there exists
$U\in\mathsf{D}(X)$ such that $x\in U$ and $U\subseteq R^{-1}(V)^{c}$.
Thus, $(U,V)\notin\mathsf{C}_{R}$, which implies that $(\varepsilon(x),\varepsilon(y))\notin R_{\mathsf{C}_{R}}$.

$(2)\Rightarrow(1)$ Suppose that for all $x,y\in X$, if $(\varepsilon(x),\varepsilon(y))\in R_{\mathsf{C}_{R}}$,
then $(x,y)\in R$. Let $y\in\text{cl}(R(x))$, but $y\notin R(x)$.
By hypothesis, we have that $(\varepsilon(x),\varepsilon(y))\notin R_{\mathsf{C}_{R}}$,
i.e. , $\varepsilon(x)\times\varepsilon(y)\nsubseteq\mathsf{C}_{R}$.
Then there exist $U,V\in\mathsf{D}(X)$ such that $x\in U$, $y\in V$,
and $(U,V)\notin\mathsf{C}_{R}$. This implies that $x\in R^{-1}(V)^{c}$,
because $x\in U\subseteq R^{-1}(V)^{c}$. Since $y\in\text{cl}(R(x))$,
we have that $R(x)\cap V\neq\emptyset$, which is a contradiction. 
\end{proof}

\begin{corollary} Let $\langle X,R\rangle$ be a Priestley precontact
space. Then 

\begin{center}$(x,y)\in R\Longleftrightarrow(\varepsilon_{X}(x),\varepsilon_{X}(y))\in R_{\mathsf{C}_{R}}$ 

\end{center}

Thus, $\varepsilon_{X}:X\longrightarrow X({\rm {D}}(X))$ is a strong
stable morphism. \end{corollary} 
\begin{proof}
We will omit the details of the proof, since all claims follow directly
from the definitions and previously established results. For the equivalence
$(x,y)\in R\text{ iff }(\varepsilon_{X}(x),\varepsilon_{X}(y))\in R_{\mathsf{C}_{R}}$,
one direction is immediate. The converse follows from Lemma \ref{caract R}
and the assumption that $(X,R)$ is a Priestley precontact space.
Finally, condition $\mathrm{(SP2)}$, which guarantees that $\varepsilon_{X}$
is a strong stable morphism, is a consequence of the surjectivity
of $\varepsilon_{X}$. 
\end{proof}

We define a functor $S:\textbf{PreDL}\longrightarrow\textbf{PS}$
as follows. For $L\in\mathbf{PreDL}$, $S(L)=\langle\text{X}(L),R_{\mathsf{C}}\rangle$.
If $h$ is a precontact homomorphism, then $S(h)=h_{*}$ is a stable
morphism. It is straightforward to see that $S$ is a well-defined
contravariant functor.

Next, we define the functor $H:\textbf{PS}\longrightarrow\textbf{PreDL}$
as follows. If $\langle X,R\rangle$ is a precontact space, then $H(X)=\langle\text{D}(X),\mathsf{C}_{R}\rangle$
is a precontact lattice, and if $f$ is a stable morphism, then $H(f)=f^{*}$
is a precontact homomorphism. Moreover, as for each $L\in\mathbf{PreDL}$,
the map $\beta_{L}:L\longrightarrow\text{D}(\text{X}(L))$ is an isomorphism
in the category $\textbf{PreDL}$, and since $(h_{*})^{*}(\beta_{L}(a))=\beta_{L}(h(a))$
for each $a\in L$, the composition functor $H\circ S:\textbf{PreDL}\longrightarrow\textbf{PreDL}$
is naturally equivalent to the identity functor. The natural equivalence
is given by the isomorphism $\beta_{L}$.

On the other hand, since for each $\langle X,R\rangle\in\mathbf{PS}$,
the map $\varepsilon_{X}:X\longrightarrow X(\text{D}(X))$ is an order-isomorphism,
a homeomorphism, and a stable morphism, it follows that $\varepsilon_{X}$
is an isomorphism in the category $\textbf{PS}$. By Priestley duality,
for each stable morphism $f:X\longrightarrow Y$, we have $(f^{*})_{*}\circ\varepsilon_{X}=\varepsilon_{Y}\circ f$.
Thus, the isomorphism $\varepsilon_{X}$ in the category $\textbf{PS}$
define a natural equivalence from the composition functor $S\circ H:\textbf{PS}\longrightarrow\textbf{PS}$
to the identity functor in $\textbf{PS}$.

Similarly, we define a functor $Q:\textbf{PreDLS}\longrightarrow\textbf{SPS}$
as follows. For $L\in\mathbf{PreDL}$, $Q(L)=\langle\text{X}(L),R_{\mathsf{C}}\rangle$,
and if $h$ is a strong precontact homomorphism, then $Q(h)=h_{*}$
is a strong stable morphism. It is straightforward to see that $Q$
is a well-defined contravariant functor.

We also define a functor $J:\textbf{SPS}\longrightarrow\textbf{PreDLS}$
as follows. If $\langle X,R\rangle\in\mathbf{SPS}$, then $J(X)=\langle\text{D}(X),\mathsf{C}_{R}\rangle\in\mathbf{PreDLS}$,
and if $f$ is a strong stable morphism, then $J(f)=f^{*}$ is a strong
precontact homomorphism.

By reasoning similarly, we can see that for each $L\in\mathbf{PreDL}$,
the map $\beta:L\longrightarrow\text{D}(\text{X}(L))$ is an isomorphism
in the category $\textbf{PreDLS}$. Since $(h_{*})^{*}(\beta_{L}(a))=\beta_{L}(h(a))$
for each $a\in L$, the composition functor $J\circ Q:\textbf{PreDLS}\longrightarrow\textbf{PreDLS}$
is naturally equivalent to the identity functor, with the natural
equivalence given by the isomorphisms $\beta_{L}$.

On the other hand, since each $\langle X,R\rangle\in\textbf{SPS}$,
the map $\varepsilon_{X}:X\longrightarrow X({\rm {D}(X))}$ is an
order-isomorphism, an homeomorphism, and a strong stable morphism,
it follows that $\varepsilon_{X}$ is an isomorphism in the category
SPS. By Priestley duality, for each strong stable morphism $f:X\longrightarrow Y$
we have that $(f^{*})_{*}\circ\varepsilon_{X}=\varepsilon_{Y}\circ f$.
Thus, the isomorphisms $\varepsilon_{X}$ of the category $\textbf{SPS}$
define a natural equivalence from the composition functor $Q\circ J:\textbf{SPS}\longrightarrow\textbf{SPS}$
to the identity in $\textbf{SPS}$. Thus, we have the following theorem.

\begin{theorem} The categories $\mathbf{PreDL}(\mathbf{PreDLS})$
and $\mathbf{PS}(\mathbf{SPS})$ are dually equivalent. \end{theorem}

\section{Correspondence theory for precontact lattices}

\label{sec:Correspondence-theory}

In this section, we characterize how various algebraic conditions
translate into properties of the relation in the associated precontact
space.

\begin{lemma} \label{caract contacto y normales} \cite{Duntsch2}\cite{Celani2}
Let $\langle L,\mathsf{C}\rangle$ be a precontact lattice. Let $\langle X,R\rangle$
the precontact space of $\langle L,\mathsf{C}\rangle$. Then: 
\begin{enumerate}
\item $\mathsf{C}$ satisfies $\mathrm{(PC4)}$ iff $R$ is reflexive. 
\item $\mathsf{C}$ satisfies $\mathrm{(PC5)}$ iff $R$ is symmetric. 
\item $\mathsf{C}$ satisfies $\mathrm{(PC6)}$ iff $R$ is transitive. 
\item The set $\{x\in L:(x,1)\notin\mathsf{C}\}$ contains only $0$ iff
$R$ is serial. 
\end{enumerate}
\end{lemma} 
\begin{proof}
$(1)$ Suppose that $\mathsf{C}$ satisfies (PC4). Let $P\in\mathrm{X}(L)$
and let $a,b\in P$. Since $P$ is a prime filter, we know that $a\wedge b\neq0$.
Then, by hypothesis, we obtain that $(a\wedge b,a\wedge b)\in C$.
Since $a\wedge b\leq a$ and $a\wedge b\leq b$, by item $(2)$ of
Lemma~\ref{propp}, it follows that $(a,b)\in C$. Thus, $R$ is
reflexive.

Suppose that $R$ is reflexive. Let $a\neq0$. Then, there exists
$Q\in\mathrm{X}(L)$ such that $a\in Q$. Since $Q\times Q\subseteq\mathsf{C}$,
we have $(a,a)\in\mathsf{C}$. Thus, $\mathsf{C}$ satisfies (PC4).

$(2)$ Assume that $\mathsf{C}$ satisfies (PC5). Let $P,Q\in{\rm X}(L)$
be such that $(P,Q)\in R$, and take $(a,b)\in Q\times P$. Since
$(b,a)\in P\times Q\subseteq\mathsf{C}$ and (PC5) holds for $\mathsf{C}$,
it follows that $(a,b)\in\mathsf{C}$.

Suppose that $R$ is symmetric. Let $(a,b)\in\mathsf{C}$. By item
$(3)$ of Lemma \ref{propp}, there exists $P,Q\in{\rm {X}}(L)$ such
that $a\in P$, $b\in Q$, and $(P,Q)\in R$. Since $R$ is symmetric,
$(Q,P)\in\mathsf{C}$. Therefore, $(b,a)\in\mathsf{C}$.

$(3)$ Suppose that $\mathsf{C}$ satisfies (PC6). Let $(P,Q)\in R$
and $(Q,D)\in R$. We consider $(p,d)\in P\times D$. If $(p,c)\notin\mathsf{C}$,
by hypothesis there exist $x,y\in L$ such that $(p,x)\notin\mathsf{C}$,
$(y,d)\notin\mathsf{C}$, and $x\vee y=1$. Since $(p,x)\notin\mathsf{C}$
and $(P,Q)\in R$, we have $x\notin Q$. Now, since $x\vee y=1$ and
$x\notin Q$, it follows that $y\in Q$. Therefore, $(y,d)\in\mathsf{C}$,
which is a contradiction. Hence, $(p,d)\in\mathsf{C}$.

Suppose that $R$ is transitive. We consider $a,b\in L$ such that
$(a,b)\notin R$. We prove that $\mathrm{Ig}(\lambda(a)\cup\alpha(b))=L$.
Suppose that $1\notin\mathrm{Ig}(\lambda(a)\cup\alpha(b))$. Then
there exists $P\in{\rm {X}}(L)$ such that $1\in P$ and $\mathrm{Ig}(\lambda(a)\cup\alpha(b))\cap P=\emptyset$.
Thus, $\lambda(a)\cap P=\emptyset$ and $\alpha(b)\cap P=\emptyset$.
Therefore, by Lemma \ref{separacion}, there exist $D,Q\in{\rm {X}}(L)$
such that $a\in Q$, $b\in D$, $(Q,P)\in R$, and $(P,D)\in R$.
Since $R$ is transitive, we have $(Q,D)\in R$ and consequently $(a,b)\in\mathsf{C}$,
which is a contradiction. Thus, there exist $x,y\in L$ such that
$(a,x)\notin\mathsf{C}$, $(y,b)\notin\mathsf{C}$, and $x\vee y=1$.

$(4)$ Assume that the set $\{x\in L:(x,1)\notin\mathsf{C}\}$ contains
only the element $0$, i.e., $\alpha(1)=\{0\}$. Let $P\in{\rm {X}}(L)$.
Since $\alpha(1)=\{0\}$, we have $\mathrm{Fg}(1)\cap\alpha^{-1}(P)=\emptyset$.
Therefore, there exists $Q\in{\rm {X}}(L)$ such that $\alpha^{-1}(P)\cap Q=\emptyset$.
Thus, $R(P)\neq\emptyset$.

Conversely, suppose that $R$ is serial. As $\mathsf{C}$ is a precontact
relation,$\{0\}\subseteq\{x:(x,1)\notin\mathsf{C}\}$. We will prove
the other inclusion. Let $x\in L$ be such that $(x,1)\notin\mathsf{C}$.
If $x\neq0$, then there exists $P\in{\rm {X}}(L)$ such that $x\in P$.
Since $R$ is serial, there exists $Q\in{\rm {X}(L)}$ such that $P\times Q\subseteq\mathsf{C}$
which contradicts the fact that $(x,1)\notin\mathsf{C}$. Therefore,
$x=0$. 
\end{proof}

\begin{remark} Given a negationally definable precontact lattice
$(L,C_{\neg})$ , the relation $C_{\neg}$ induces on the dual space
$X(L)$ a relation $R_{C_{\neg}}$ which coincides with the relation
$R_{\neg}$ associated with the negation operator $\neg$ (see \cite[Section 4]{Celani6}
and Example \ref{ex neg}). Consequently %any topological or relational
%property of $R_{\neg}$ is reflected in $R_{C_{\neg}}$, and vice
%versa. In particular, 
if $R_{\neg}$ is an equivalence relation, then so is $R_{C_{\neg}}$.
It follows that the normal negationally definable precontact lattices
are precisely those $\neg$-lattices $(L,\neg)$ for which $R_{\neg}$
is an equivalence relation. As shown in \cite{Celani6}, this condition
characterizes the class of quasi-Stone algebras. Hence, quasi-Stone
algebras are in bijective correspondence with the normal negationally
definable precontact lattices. \end{remark}

We now turn to a relational description of compingent lattices, that
is , normal lattices satisfying the condition (PC7). The approach
mirrors the Boolean setting, where \cite{Bezha1} establishes that
compingent algebras are precisely those whose Stone dual spaces are
equipped with an equivalence relation $R$ satisfying the following
irreducibility condition: for every proper closed subset $Y$ of $X$,
the image $R(Y)$ is a proper closed subset of $X$.

\begin{definition}Let $\left<X,R\right>$ be a precontact space.
Let $R$ be an equivalence relation. We shall say that $R$ is an
order-irreducible equivalence relation if for every proper closed
upset $Y$ of $X$, the set $R(Y)$ is a proper closed downset.

\end{definition}

\begin{theorem} \label{caract compingentes} Let $\langle L,\mathsf{C}\rangle$
be a normal lattice, and let $\langle X,R\rangle$ be the dual of
$\langle L,\mathsf{C}\rangle$. Then the following conditions are
equivalent: 
\begin{enumerate}
\item $R$ is an order-irreducible equivalence relation. 
\item $\langle L,\mathsf{C}\rangle$ satisfies the condition $\mathrm{(PC7)}.$ 
\end{enumerate}
\end{theorem} 
\begin{proof}
$(1)\Rightarrow(2)$ Suppose that for every proper closed upset $Y$
of $X$, the set $R(Y)$ is a proper closed downset. Let $a\in L$
with $a\neq1$. Then $\beta(a)\subset X$, and since $\beta(a)$ is
a proper closed upset, it follows by hypothesis that $R(\beta(a))$
is a proper closed downset. So, there exists $P\in X$ such that $P\notin R(\beta(a))$.
Since $R(\beta(a))$ is closed and downset, there exists a closed
upset $U\subseteq X$ such that $P\in U$ and $U\cap R(\beta(a))=\emptyset$.
By Priestley duality, $U=\beta(b)$ for some $b\in L$. Since $P\in\beta(b)$,
we have that $b\neq0$. Moreover, since $\beta(b)\cap R(\beta(a))=\emptyset$,
it follows that $(a,b)\notin\mathsf{C}$. This proves that for every
$a\neq1$, there exists $b\neq0$ such that $(a,b)\notin\mathsf{C}$.

$(2)\Rightarrow(1)$ Suppose that $\langle L,\mathsf{C}\rangle$ satisfies
the condition ($\mathrm{PC7)}.$ Let $Y\subseteq X$ be a proper closed
upset. By Lemma \ref{Caracterizacion de espacio}, the set $R(Y)$
is a closed downset. It remains to show that $R(Y)$ is proper. Since
$Y$ is proper, there exists $P\in X$ such that $P\notin Y$. Because
$X$ is a Priestley space, there exists a clopen upset $U\subseteq X$
such that $Y\subseteq U$ and $P\notin U$. By Priestley duality,
$U=\beta(a)$ for some $a\in L$. Since $P\notin\beta(a)$, we have
$a\neq1$. By hypothesis, there exists $b\neq0$ such that $(a,b)\notin\mathsf{C}$.
This implies that $R(\beta(b))\cap Y=\emptyset$. As $R$ is an equivalence
relation, this yields $\beta(b)\cap R(Y)=\emptyset$. Thus $R(Y)\subseteq\beta(b)^{c}$.
Because $b\neq0$, we know that $\beta(b)^{c}\subsetneq X$. Hence,
$R(Y)$ is a proper subset of $X$, as desired. 
\end{proof}

\begin{corollary} Let $\langle L,\mathsf{C}\rangle$ be a precontact
lattice. Then: 
\begin{enumerate}
\item $\langle L,\mathsf{C}\rangle$ is a contact lattice if and only if
$R$ is reflexive and symmetric. 
\item $\langle L,\mathsf{C}\rangle$ is a normal lattice if and only if
$R$ is an equivalence relation. 
\item $\langle L,\mathsf{C}\rangle$ is a compingent lattice if and only
if $R$ is an order-irreducible equivalence relation. 
\end{enumerate}
\end{corollary} 
\begin{proof}
The proof follows immediately from Lemma \ref{caract contacto y normales}
and Theorem \ref{caract compingentes}. 
\end{proof}

\section{Precontact substructures and strong precontact sublattices}

\label{sec:Subestructures-and-subalgebras}

In this section, we introduce the notions of precontact substructures
and strong precontact sublattices. We show that every strong precontact
sublattice is a precontact substructure, while the converse does not
necessarily hold. Finally, we use duality to characterize both precontact
substructures and strong precontact sublattices in terms of relational
conditions on the corresponding Priestley dual spaces.

In \cite{CigLafPet}, the notion of a lattice preorder relation on
a Priestley space was introduced. It was also shown that there exists
a dual isomorphism between the lattice of bounded sublattices of $L$
and the lattice of lattice preorder relations on the Priestley space
of $L$. Formally, a relation $R$ on a Priestley space $\langle X,\leq,\tau\rangle$
is a lattice preorder if it is reflexive, transitive, and satisfies
the following condition: 
\begin{flushleft}
$\forall x,y\in X,\hspace{0.1cm}\text{if}\hspace{0.1cm}(x,y)\notin R\hspace{0.1cm}\text{then }\hspace{0.1cm}\exists U\in{\rm {D}}(X)$
such that $R^{-1}(U)\subseteq U,x\in U$ and $y\notin U$. 
\par\end{flushleft}

In particular, given a bounded distributive lattice $L$ and the domain
$M$ of a bounded sublattice of $L$, it was shown in \cite{CigLafPet}
that: 
\[
S_{M}=\{(P,Q)\in{\rm {X}}(L)\times{\rm {X}}(L):Q\cap M\subseteq P\}
\]
is a lattice preorder on the associated Priestley space $\langle{\rm {X}}(L),\subseteq,\tau_{L}\rangle$.

Conversely, if $R$ is a lattice preorder on a Priestley space$\langle X,\leq,\tau\rangle$,
then 
\[
M_{R}=\{U\in{\rm {D}}(X):R^{-1}(U)\subseteq U\}
\]
is the domain of a bounded sublattice of ${\rm {D}}(X)$.

We will often work with sublattices and precontact lattice of a given
lattice. To avoid any confusion, if $L$ and $M$ are two precontact
lattices, we will use $\mathsf{C}_{L}$ to denote the precontact relation
corresponding to $L$, and similarly $\mathsf{C}_{M}$ for $M$. %The ideal generated
%in $L$ by a subset $A\subseteq L$ will be denoted by $\mathrm{Ig}_{L}(A)$,
%while the filter generated by $A$ in $L$ will be denoted by $\mathrm{Fg}_{L}(A)$.
When $L$ is a bounded distributive lattice and $M$ is a bounded
sublattice (or substructure) of $L$, the maximum and minimum elements
of $L$ and $M$ are the same, denoted by $1$ and $0$ respectively.
We will also write $\alpha_{L}$ and $\alpha_{M}$ to distinguish
between the sets associated with $L$ and $M$, respectively. If $B\subseteq L$
then $\alpha^{-1}_{L}(B)=\{a\in L:\alpha_{L}(a)\cap B\neq\emptyset\}$
and for $B\subseteq M$, we write $\alpha^{-1}_{M}(B)=\{a\in M:\alpha_{L}(a)\cap B\neq\emptyset\}$.
This notation that will be particularly relevant in the proof of Proposition
\ref{caract prec}.

Let $\langle L,\mathsf{C}\rangle$ be a precontact lattice and let
$M$ be a bounded sublattice of $L$. Let $\mathsf{C}_{M}=\mathsf{C}_{L}\cap(M\times M)$.
It can be verified that $\mathsf{C}_{M}$ is a precontact relation
on $M$.

\begin{definition} Let $\left<L,\mathsf{C}_{L}\right>$ and $\left<M,\mathsf{C}_{M}\right>$
be a precontact lattices. Let $M$ be a bounded sublattice of $L$.
We shall say that $\left<M,\mathsf{C}_{M}\right>$ is a precontact
substructure of $\left<L,\mathsf{C}_{L}\right>$ if $\mathsf{C}_{M}=\mathsf{C}_{L}\cap\left(M\times M\right)$.
\end{definition}

The following result, which is a generalization of Theorem $4$ from
\cite{Duntsch4}, establishes how the precontact substructures of
the precontact lattices are characterized in terms of their dual spaces.

\begin{proposition} \label{sub dual} Let $\left<L,\mathsf{C}_{L}\right>$
and $\left<M,\mathsf{C}_{M}\right>$ be precontact lattices such that
$M$ is a $(0,1)$-sublattice of $L$. Then the following conditions
are equivalents: 
\begin{enumerate}
\item $\left<M,\mathsf{C}_{M}\right>$ is a precontact substructure of $\left<L,\mathsf{C}_{L}\right>$ 
\item $R_{\mathsf{C}_{M}}=\{(P,Q)\in\text{X}(M)^{2}\mid\exists P'\in{\rm {X}(L)}\exists Q'\in{\rm {X}(L)}\text{ such that }P=P'\cap M,Q=Q'\cap M,\text{ y }(P',Q')\in S_{M}\circ R_{\mathsf{C}_{L}}\circ S^{-1}_{M}\}$ 
\end{enumerate}
\end{proposition} 
\begin{proof}
$(1)\Rightarrow(2)$ Suppose that $\left<M,\mathsf{C}_{M}\right>$
is a precontact substructure of $\left<L,\mathsf{C}_{L}\right>$,
i.e., $\mathsf{C}_{M}=\mathsf{C}_{L}\cap(M\times M)$. Let $P,Q\in{X}(M)$
be such that $(P,Q)\in R_{\mathsf{C}_{M}}$. Since $\mathrm{Fg}(P)$
and $\mathrm{Fg}(Q)$ are proper filters of $L$ and $\mathrm{Fg}(P)\times\mathrm{Fg}(Q)\subseteq\mathsf{C}_{L}$,
by item $(6)$ of the Lemma \ref{propp}, we have that there exist
$F,G\in\mathrm{X}(L)$ such that $\mathrm{Fg}(P)\subseteq F$, $\mathrm{Fg}(Q)\subseteq G$,
and $(F,G)\in R_{\mathsf{C}_{L}}$. On the other hand, since $P,Q\in\mathrm{X}(M)$,
there exist $P'\in\mathrm{X}(L)$ and $Q'\in\mathrm{X}(L)$ such that
$P=P'\cap M$ and $Q=Q'\cap M$. As $P\subseteq F$ and $Q\subseteq G$,
we have $P'\cap M\subseteq F$ and $Q'\cap M\subseteq G$. From this,
it follows that $(F,P')\in S_{M}$ and $(G,Q')\in S_{M}$. Thus, $(P',Q')\in S_{M}\circ R_{\mathsf{C}_{L}}\circ S^{-1}_{M}$.
Now, let $P,Q\in\mathrm{X}(M)$ be such that there exist $P',Q'\in\mathrm{X}(L)$
with $P=P'\cap M$, $Q=Q'\cap M$, and $(P',Q')\in S_{M}\circ R_{\mathsf{C}_{L}}\circ S^{-1}_{M}$.
We will prove that $(P,Q)\in R_{\mathsf{C}_{M}}$. Let $(a,b)\in P\times Q$.
Since $(P',Q')\in S_{M}\circ R_{\mathsf{C}_{L}}\circ S^{-1}_{M}$,
there exist $G\in\mathrm{X}(L)$ and $F\in\mathrm{X}(L)$ such that
$(G,Q')\in S_{M}$, $(P',F)\in S^{-1}_{M}$, and $(F,G)\in R_{\mathsf{C}_{L}}$.
Thus, $P=P'\cap M\subseteq F$ and $Q=Q'\cap M\subseteq G$. This
implies that $(a,b)\in F\times G\subseteq\mathsf{C}_{L}$. Hence,
$(a,b)\in\mathsf{C}_{L}\cap(M\times M)=\mathsf{C}_{M}$.

$(2)\Rightarrow(1)$ First, we prove that $\mathsf{C}_{M}\subseteq\mathsf{C}_{L}\cap(M\times M)$.
Let $(a,b)\in\mathsf{C}_{M}$. By the item $(7)$ of the Lemma \ref{propp},
there exist $P,Q\in\mathrm{X}(M)$ such that $a\in P$, $b\in Q$,
and $(P,Q)\in R_{\mathsf{C}_{M}}$. Since $(P,Q)\in R_{\mathsf{C}_{M}}$,
by hypothesis, there exist $P',Q'\in\mathrm{X}(L)$ such that $P=P'\cap M$,
$Q=Q'\cap M$, and $(P',Q')\in S_{M}\circ R_{\mathsf{C}_{L}}\circ S^{-1}_{M}$.
Then, there exist $G,F\in\mathrm{X}(L)$ such that $Q'\cap M\subseteq G$,
$P'\cap M\subseteq F$, and $F\times G\subseteq\mathsf{C}_{L}$.{}
Since $a\in P\subseteq F$ and $b\in Q\subseteq G$, we have $(a,b)\in\mathsf{C}_{L}\cap(M\times M)$.
Therefore, $\mathsf{C}_{M}\subseteq\mathsf{C}_{L}\cap(M\times M)$.
To prove the other inclusion, let $a,b\in M$ be such that $(a,b)\in\mathsf{C}_{L}$.
By item (7) of Lemma~\ref{propp}, there exist $P',Q'\in\mathrm{X}(L)$
such that $(a,b)\in P'\times Q'$ and $(P',Q')\in R_{\mathsf{C}_{L}}$.
Observe that $(P',Q')\in S_{M}\circ R_{\mathsf{C}_{L}}\circ S^{-1}_{M}$,
since $(P',P')\in S^{-1}_{M}$ and $(Q',Q')\in S_{M}$. Let us define
$P=P'\cap M$ and $Q=Q'\cap M$. By hypothesis, it follows that $(P,Q)\in R_{\mathsf{C}_{M}}$.
Since $(a,b)\in P\times Q\subseteq R_{\mathsf{C}_{M}}$, we conclude
that $(a,b)\in\mathsf{C}_{M}$. 
\end{proof}

We now introduce the previously mentioned stronger notion, which extends
the concept of a precontact substructure.

\begin{definition} Let $\langle L,\mathsf{C}_{L}\rangle$ and $\langle M,\mathsf{C}_{M}\rangle$
be precontact lattices, where $M$ is a bounded sublattice of $L$.
We say that $\langle M,\mathsf{C}_{M}\rangle$ is a strong precontact
sublattice of $\langle L,\mathsf{C}_{L}\rangle$ if the following
conditions hold: 
\begin{enumerate}
\item[$\mathrm{(C1)}$] $\mathsf{C}_{L}\cap(M\times M)=\mathsf{C}_{M}$. 
\item[$\mathrm{(C2)}$] For each $a\in L$ and $b\in M$, if $(a,b)\notin\mathsf{C}_{L}$,
there exists $c\in M$ such that $(c,b)\notin\mathsf{C}_{M}$ and
$a\leq c$. 
\end{enumerate}
\end{definition}

An alternative but dual condition to $\mathrm{(C2)}$ is the following: 
\begin{enumerate}
\item[$\mathrm{(C3)}$] For each $a\in L$ and $b\in M$, if $(b,a)\notin\mathsf{C}_{L}$,
there exists $c\in M$ such that $(b,c)\notin\mathsf{C}_{M}$ and
$a\leq c$. 
\end{enumerate}
These conditions reflect how the precontact relation interacts with
the sublattice. When $\mathsf{C}_{L}$ is symmetric, they coincide,
making either one sufficient to define a strong precontact sublattice.
In this work, we focus on (C2), although a dual approach based on
(C3) is equally valid.

Clearly, every strong precontact sublattice is a precontact substructure.
The following example shows that the converse is not necessarily true.
\begin{example}
Consider the lattice $L=\{0,b,c,d,1\}$ such that: 
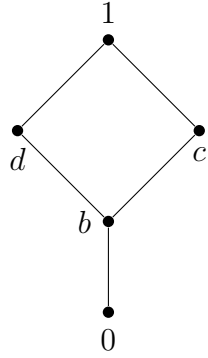
\begin{center}
\begin{tikzpicture}[scale=1.2]

% Nodos
\node [circle, fill=black, inner sep=1.5pt, label=above:$1$]  (1) at (0,2) {};
\node [circle, fill=black, inner sep=1.5pt, label=below:$d$] (d) at (-1,1) {};
\node [circle, fill=black, inner sep=1.5pt, label=below:$c$](c) at (1,1) {};
\node [circle, fill=black, inner sep=1.5pt, label=left:$b$](b) at (0,0) {};
\node [circle, fill=black, inner sep=1.5pt, label=below:$0$](0) at (0,-1) {};

% Aristas
\draw (0) -- (b);
\draw (b) -- (d);
\draw (b) -- (c);
\draw (d) -- (1);
\draw (c) -- (1);

\end{tikzpicture} 
\par\end{center}
and the precontact relation on $L$ given by: 
\[
\mathsf{C}_{L}=\begin{aligned}\{(d,b),(d,d),(d,c),(d,1),(1,b),(1,c),(1,d),(1,1)\}\end{aligned}
\]
It can be verified that this relation is a precontact relation in
$L$. Now consider the sublattice $M=\{0,d,1\}$ and the relation
\[
\mathsf{C}_{M}=\mathsf{C}_{L}\cap(M\times M)=\{(d,1),(1,d),(d,d),(1,1)\}.
\]
Note that $(b,d)\notin\mathsf{C}_{L}$ but for any $t\in M$ with
$b\leq t$, we have $(t,d)\in\mathsf{C}_{M}$. Thus, $\langle M,\mathsf{C}_{M}\rangle$
is a precontact substructure of $L$ but not is a strong precontact
sublattice. 
\end{example}

\begin{remark} We can see that the notion of a strong precontact
sublattice appropriately extends the concept of a subalgebra in the
case of modal algebras. That is, $S$ is a subalgebra of a modal Boolean
algebra $(B,\square)$ iff $S$ is also a strong precontact subalgebra
of the precontact algebra $(B,\mathsf{C}_{\square})$ (see Example
\ref{modal-prec-unified}). Indeed, suppose that $M$ is a subalgebra
of $(B,\square)$. Let $\mathsf{C}_{M}=\mathsf{C}_{\square}\cap(M\times M)$.
If $a\in B$ and $b\in M$ are such that $(a,b)\notin\mathsf{C}_{\square}$,
then $a\leq\square\neg b$. Since $\neg b\in M$ and $M$ is a modal
subalgebra, it follows that $\square\neg b\in M$. Furthermore, $(\square\neg b,b)\notin{\mathsf{C}}_{M}$
because if $(\square\neg b,b)\in\mathsf{C}_{M}$, then $(\square\neg b,b)\in\mathsf{C}_{\square}$,
which implies that $\square\neg b\nleq\square\neg b$. Therefore,
$M$ is a strong precontact subalgebra of $(B,\mathsf{C}_{\square})$.

Now, suppose that $M$ is a strong precontact subalgebra of $(B,\mathsf{C}_{\square})$.
Let $b\in M$. Then $\neg b\in M$. Since $(\square b,\neg b)\notin\mathsf{C}_{\square}$,
$\square b\in B$, and $M$ is a precontact subalgebra, there exists
$c\in M$ such that $(c,\neg b)\notin\mathsf{C}_{M}$ and $\square b\leq c$.
Thus, $\square b\leq c\leq\square b$, and since $c\in M$, it follows
that $\square b\in M$. \end{remark}

In universal algebra it is well known that a subset $M$ of a lattice
$L$ is a sublattice if and only if the inclusion map $i\colon M\hookrightarrow L$,
defined by $i(m)=m$ for all $m\in M$, is a lattice homomorphism.
The following theorem shows that this characterization extends naturally
to the setting of precontact lattices. More precisely, we prove that
precontact substructures (and strong precontact sublattices) can be
equivalently described in terms of the corresponding inclusion map
being a precontact homomorphism (respectively, a strong precontact
homomorphism).

\begin{theorem} Let $\langle L,\mathsf{C}_{L}\rangle$ and $\langle M,\mathsf{C}_{M}\rangle$
be precontact lattices such that $M\subseteq L$ is a bounded sublattice,
and let $i:M\hookrightarrow L$ be the inclusion map. Then: 
\begin{itemize}
\item[(i)] $\langle M,\mathsf{C}_{M}\rangle$ is a precontact substructure of
$\langle L,\mathsf{C}_{L}\rangle$ if and only if $i$ is a precontact
homomorphism. 
\item[(ii)] $\langle M,\mathsf{C}_{M}\rangle$ is a strong precontact sublattice
of $\langle L,\mathsf{C}_{L}\rangle$ if and only if $i$ is a strong
precontact homomorphism. 
\end{itemize}
\end{theorem}
\begin{proof}
(i) Assume that $\langle M,\mathsf{C}_{M}\rangle$ is a precontact
substructure of $\langle L,\mathsf{C}_{L}\rangle$, and we prove that
the inclusion map $i$ is a precontact homomorphism. Let $a,b\in M$
be such that $(i(a),i(b))\in\mathsf{C}_{L}$. Since $a,b\in M$, we
have $i(a)=a$ and $i(b)=b$. Therefore, $(a,b)\in\mathsf{C}_{L}\cap(M\times M)=\mathsf{C}_{M}$,
because $\langle M,\mathsf{C}_{M}\rangle$ is a precontact substructure
of $\langle L,\mathsf{C}_{L}\rangle$.

Now suppose that $i$ is a precontact homomorphism. Let $(a,b)\in\mathsf{C}_{L}\cap(M\times M)$.
Then $(a,b)=(i(a),i(b))\in\mathsf{C}_{L}$, and since $i$ is a precontact
homomorphism, it follows that $(a,b)\in\mathsf{C}_{M}$. Hence, $\mathsf{C}_{L}\cap(M\times M)\subseteq\mathsf{C}_{M}$.
We now prove the other inclusion. Let $(a,b)\in\mathsf{C}_{M}$. Then,
by $(7)$ of Lemma \ref{propp}, there exist $P,Q\in\mathrm{X}(M)$
such that $a\in P$, $b\in Q$, and $(P,Q)\in R_{\mathsf{C}_{M}}$.
By Proposition \ref{sub dual}, there exist $P',Q'\in\mathrm{X}(L)$
such that $P=P'\cap M$, $Q=Q'\cap M$, and $(P',Q')\in S_{M}\circ R_{\mathsf{C}_{L}}\circ S^{-1}_{M}$.
So, there exist $F,G\in\mathrm{X}(L)$ such that $P'\cap M\subseteq F$,
$(F,G)\in R_{\mathsf{C}_{L}}$, and $Q'\cap M\subseteq G$. Since
$a\in F$, $b\in G$, and $F\times G\subseteq\mathsf{C}_{L}$, it
follows that $(a,b)\in\mathsf{C}_{L}$. Therefore, $(a,b)\in\mathsf{C}_{L}\cap(M\times M)$,
and hence $\mathsf{C}_{M}\subseteq\mathsf{C}_{L}\cap(M\times M)$.

$(ii)$ The proof is straightforward and follows directly from the
definitions of strong precontact homomorphism and strong precontact
sublattice. 
\end{proof}

Given a precontact lattice $\langle L,\mathsf{C}_{L}\rangle$ we provide
a characterization of its strong precontact sublattice which will
be used in the proof of Proposition~\ref{caract prec}.

\begin{lemma}\label{caract subret alpha} Let $\langle L,\mathsf{C}_{L}\rangle$
and $\langle M,\mathsf{C}_{M}\rangle$ be precontact lattices such
that $M$ be a $(0,1)$-sublattice of $L$. Then the following conditions
are equivalent: 
\begin{enumerate}
\item $\langle M,\mathsf{C}_{M}\rangle$ is a strong precontact sublattice
of $\langle L,\mathsf{C}_{L}\rangle$ 
\item $\mathrm{Ig}(M\setminus\mathsf{C}^{-1}_{M}(x))=L\setminus\mathsf{C}^{-1}_{L}(x)$
for each $x\in M$. 
\end{enumerate}
\end{lemma} 
\begin{proof}
Let $\langle L,\mathsf{C}_{L}\rangle$ and $\langle M,\mathsf{C}_{M}\rangle$
be precontact lattices. Let $M$ be a $(0,1)$-sublattice of $L$.

$(1)\Rightarrow(2)$ Suppose that $\langle M,\mathsf{C}_{M}\rangle$
is a strong precontact sublattice of $\langle L,\mathsf{C}_{L}\rangle$,
and let $x\in M$. It is immediate that $\mathrm{Ig}_{L}(M\setminus\mathsf{C}^{-1}_{L}(x))\subseteq L\setminus\mathsf{C}^{-1}_{L}(x)$.
We now show the other inclusion. Take $a\in L\setminus\mathsf{C}^{-1}_{L}(x)$.
By hypothesis, there exists $c\in M$ with $(c,x)\notin\mathsf{C}_{M}$
and $a\leq c$. Therefore $a\in\mathrm{Ig}(M\setminus\mathsf{C}^{-1}_{M}(x))$.

$(2)\Rightarrow(1)$ Suppose that $\mathrm{Ig}_{L}(M\setminus\mathsf{C}^{-1}_{M}(x))=L\setminus\mathsf{C}^{-1}_{L}(x)$
for each $x\in M$. We prove that $\langle M,\mathsf{C}_{M}\rangle$
is a strong precontact sublattice of $\langle L,\mathsf{C}_{L}\rangle$.
To see this, we first show that $M$ is a precontact substructure
of $L$. Let $(a,b)\in\mathsf{C}_{M}$. Thus, $a,b\in M$. If $(a,b)\notin\mathsf{C}_{L}$,
then $a\in\mathrm{Ig}(M\setminus\mathsf{C}^{-1}_{M}(b))$. Therefore,
there exists $c\in M\setminus\mathsf{C}^{-1}_{M}(b)$ such that $a\leq c$.
Since $M\setminus\mathsf{C}^{-1}_{M}(b)$ is an ideal, it follows
that $a\in M\setminus\mathsf{C}^{-1}_{M}(b)$, which is a contradiction.
Thus, we have $\mathsf{C}_{M}\subseteq\mathsf{C}_{L}\cap(M\times M)$.
To prove the other inclusion, consider $(a,b)\in\mathsf{C}_{L}\cap(M\times M)$.
If $(a,b)\notin\mathsf{C}_{M}$, then $a\in M\setminus\mathsf{C}^{-1}_{M}(b)\subseteq\mathrm{Ig}_{L}(M\setminus\mathsf{C}^{-1}_{M}(b)=L\setminus\mathsf{C}^{-1}_{L}(b)$,
which is a contradiction.

Now, we prove that for all $a\in L$ and $b\in M$, if $(a,b)\notin\mathsf{C}_{L}$,
then there exists $c\in M$ such that $(c,b)\notin\mathsf{C}_{M}$
and $a\leq c$. Let $a\in L$ and $b\in M$ be such that $(a,b)\notin\mathsf{C}_{L}$.
Thus, $a\in L\setminus\mathsf{C}^{-1}_{L}(b)$. Since $L\setminus\mathsf{C}^{-1}_{L}(b)=\mathrm{Ig}(M\setminus\mathsf{C}^{-1}_{M}(x))$,
there exists $c\in M$ such that $(c,b)\notin\mathsf{C}_{M}$ and
$a\leq c$. 
\end{proof}

\begin{remark} Observe that $\alpha_{M}(x)\coloneqq M\setminus\mathsf{C}^{-1}_{M}(x)$
and $\alpha_{L}(x)\coloneqq L\setminus\mathsf{C}^{-1}_{L}(x)$, the
previous result can be reformulated in terms of the operator $\alpha$
as follows: $\langle M,\mathsf{C}_{M}\rangle$ is a strong precontact
sublattice of $\langle L,\mathsf{C}_{L}\rangle$ if and only if $\mathrm{Ig}(\alpha_{M}(x))=\alpha_{L}(x)$
for each $x\in M$.

To characterize the strong precontact sublattice $\langle M,\mathsf{C}_{M}\rangle$
of a precontact lattice $\langle L,\mathsf{C}_{L}\rangle$ in terms
of the operator $\lambda$, it is necessary to adopt the dual definition
of strong precontact sublattice. Specifically, this requires replacing
conditions (C1) and (C2) with (C1) and (C3).

\end{remark}

In what follows, we provide a characterization of strong precontact
sublattice in terms of relational conditions in the associated dual
space. As mentioned earlier, the notion of a strong precontact sublattice
refines that of a precontact substructure by imposing an additional
requirement. The next proposition shows that this condition admits
a purely relational formulation in the dual setting. This result extends
the correspondence between precontact lattice and preorder relations
explored in \cite{CigLafPet} to the enriched setting of precontact
lattices.

\begin{proposition} \label{caract prec} Let $\langle L,\mathsf{C}_{L}\rangle$
be a precontact lattice and let $\langle M,\mathsf{C}_{M}\rangle$
be a precontact substructure of $\langle L,\mathsf{C}_{L}\rangle$.
Then the following conditions are equivalent: 
\begin{enumerate}
\item For each $a\in L$ and $b\in M$, if $(a,b)\notin\mathsf{C}_{L}$,
there exists $c\in M$ such that $(c,b)\notin\mathsf{C}_{M}$ and
$a\leq c$. 
\item $R_{\mathsf{C}_{L}}\circ S^{-1}_{M}\subseteq S_{M}\circ R_{\mathsf{C}_{L}}$. 
\end{enumerate}
\end{proposition} 
\begin{proof}
Let $\langle L,\mathsf{C}_{L}\rangle$ be a precontact lattice and
let $\langle M,\mathsf{C}_{M}\rangle$ be a precontact substructure
of $\langle L,\mathsf{C}_{L}\rangle$.

\noindent$(1)\Rightarrow(2)$ Assume that for each $a\in L$ and
$b\in M$, if $(a,b)\notin\mathsf{C}_{L}$, then there exist $c\in M$
such that $(c,b)\notin\mathsf{C}_{M}$ and $a\leq c$. Let $P,Q,D\in\mathrm{X}(L)$
such that $Q\cap M\subseteq P$ and $P\times D\subseteq\mathsf{C}_{L}$.
We will show that $\alpha^{-1}_{L}(Q)\cap\mathrm{Fg}(D\cap M)=\emptyset$.
Suppose, for contradiction, that $\alpha^{-1}_{L}(Q)\cap\mathrm{Fg}(D\cap M)\neq\emptyset$.
Then there exist $x\in\alpha^{-1}_{L}(Q)$ and $b\in D\cap M$ such
that $b\leq x$. Since $\alpha^{-1}_{L}(Q)$ is an ideal and $x\in\alpha^{-1}_{L}(Q)$,
it follows that $b\in\alpha^{-1}_{L}(Q)$. Hence, there exists $a\in Q$
such that $(a,b)\notin\mathsf{C}_{L}$. This implies, by hypothesis,
that there exists $c\in M$ such that $(c,b)\notin\mathsf{C}_{M}$
and $a\leq c$. Since $a\in Q$ and $a\leq c$, it follows that $c\in Q$.
So, $c\in Q\cap M\subseteq P$. Therefore, $(c,b)\in P\times D\subseteq\mathsf{C}_{L}$.
But since $c,b\in M$ and $\langle M,\mathsf{C}_{M}\rangle$ is a
precontact substructure of $\langle L,\mathsf{C}_{L}\rangle$, it
must be the case that $(c,b)\in\mathsf{C}_{M}$, which is a contradiction.
Hence, $\alpha^{-1}_{L}(Q)\cap\mathrm{Fg}(D\cap M)=\emptyset$. Therefore,
there exists $Z\in\mathrm{X}(L)$ such that $(Q,Z)\in R_{\mathsf{C}_{L}}$
and $(Z,D)\in S_{M}$, so we conclude that $R_{\mathsf{C}_{L}}\circ S^{-1}_{M}\subseteq S_{M}\circ R_{\mathsf{C}_{L}}$.

$(2)\Rightarrow(1)$. Assume that $R_{\mathsf{C}_{L}}\circ S^{-1}_{M}\subseteq S_{M}\circ R_{\mathsf{C}_{L}}$.
According to Lemma \ref{caract subret alpha}, proving that $\langle M,\mathsf{C}_{M}\rangle$
is a strong precontact sublattice of $\langle L,\mathsf{C}_{L}\rangle$
is equivalent to proving that $\mathrm{Ig}_{L}(M\setminus\mathsf{C}^{-1}_{L}(b))=L\setminus\mathsf{C}^{-1}_{L}(b)$
for all $b\in M$. Therefore, we will focus on proving this equality.
Let $b\in M$. Clearly, $\mathrm{Ig}_{L}(M\setminus\mathsf{C}^{-1}_{L}(b))\subseteq L\setminus\mathsf{C}^{-1}_{L}(b)$.
Suppose the converse inclusion fails. Then there exists $x\in L\setminus\mathsf{C}^{-1}_{L}(b)$
such that $x\nleq y$ for all $y\in M\setminus\mathsf{C}^{-1}_{L}(b)$.
This implies that $\mathrm{Fg}(\mathrm{Fg}(x)\cap M)\cap(L\setminus\mathsf{C}^{-1}_{L}(b))=\emptyset$.
By the Prime Filter Theorem, there exists $D\in\mathrm{X}(L)$ such
that $\mathrm{Fg}(\mathrm{Fg}(x)\cap M)\subseteq D$ and $(L\setminus\mathsf{C}^{-1}_{L}(b))\cap D=\emptyset$.
Hence, $b\notin\alpha^{-1}_{L}(D)$. Then, by Lemma \ref{separacion},
there exists $Q\in\mathrm{X}(L)$ such that $(D,Q)\in R_{\mathsf{C}_{L}}$
and $b\in Q$. We now show that $\mathrm{Fg}(x)\cap\mathrm{Ig}(M\cap D^{c})=\emptyset$.
Suppose the intersection is nonempty. Then there exists $t\in L$
such that $x\leq t\leq w$ for some $w\in M\cap D^{c}$. Thus, $w\in\mathrm{Fg}(\mathrm{Fg}(x)\cap M)$,
but $w\notin D$, contradicting $\mathrm{Fg}(\mathrm{Fg}(x)\cap M)\subseteq D$.
Therefore, $\mathrm{Fg}(x)\cap\mathrm{Ig}(M\cap D^{c})=\emptyset$.
Again, by the Prime Filter Theorem, there exists $P\in\mathrm{X}(L)$
such that $x\in P$ and $P\cap M\subseteq D$. Hence, $(P,D)\in S^{-1}_{M}$.
Since $(D,Q)\in R_{\mathsf{C}_{L}}$, it follows that $(P,Q)\in R_{\mathsf{C}_{L}}\circ S^{-1}_{M}\subseteq S_{M}\circ R_{\mathsf{C}_{L}}$.
Then there exists $Z\in\mathrm{X}(L)$ such that $(P,Z)\in R_{\mathsf{C}_{L}}$
and $Q\cap M\subseteq Z$. Since $b\in Q\cap M$, we have $b\in Z$,
and since $x\in P$, it follows that $(x,b)\in\mathsf{C}_{L}$, contradicting
the assumption. We conclude that $L\setminus\mathsf{C}^{-1}_{L}(b)=\mathrm{Ig}_{L}((L\setminus\mathsf{C}^{-1}_{L}(b))\cap M)$.
Thus, there exists $c\in M$ such that $(c,b)\notin\mathsf{C}_{M}$
and $a\leq c$, as desired. 
\end{proof}

\begin{corollary}\label{corresp} Let $\langle L,\mathsf{C}\rangle$
be a precontact lattice. The correspondence $M\mapsto S_{M}$ establishes
an anti-isomorphism from the lattice of strong precontact sublattices
of $L$ onto the lattice of lattice preorders of ${\rm {X}(L)}$ such
that $R_{\mathsf{C}_{L}}\circ S^{-1}_{M}\subseteq S_{M}\circ R_{\mathsf{C}_{L}}$.
\end{corollary}

We remember that every lattice with a negation operator $\langle L,\neg\rangle$
induces a precontact lattice $\langle L,\mathsf{C}_{\neg}\rangle$
(see Example \ref{ex neg}). Recall that a sublattice $M$ of a lattice
with a negation operator $\langle L,\neg\rangle$ is a $(0,1)$-sublattice
$M$ that is closed under the negation operator, i.e., for all $a\in M$,
we have $\neg a\in M$. In the following lemma, we establish a correspondence
between sublattices of a lattice with a negation operator $\langle L,\neg\rangle$
and strong precontact sublattice of the associated precontact lattice
$\langle L,\mathsf{C}_{\neg}\rangle$.

\begin{lemma}\label{sublat} Let $\langle L,\neg\rangle$ be a lattice
with a negation operator and let $M$ be a $(0,1)$-sublattice of
$L$. Then the following conditions are equivalent: 
\begin{enumerate}
\item $M$ is a sublattice of $\langle L,\neg\rangle$. 
\item $\langle M,\mathsf{C}_{M}\rangle$ is a strong precontact sublattice
of $\langle L,\mathsf{C}_{\neg}\rangle$ where $\mathsf{C}_{M}=\mathsf{C}_{\neg}\cap(M\times M)$. 
\end{enumerate}
\end{lemma} 
\begin{proof}
$(1)\Rightarrow(2)$ Suppose that $M$ is a sublattice of $\langle L,\neg\rangle$.
Let $\mathsf{C}_{M}=\mathsf{C}_{\neg}\cap(M\times M)$. Let $a\in L$
and $b\in M$ such that $(a,b)\notin\mathsf{C}_{\neg}$. Thus, by
definition of $\mathsf{C}_{\neg}$, $a\leq\neg b$. Since $M$ is
a sublattice of $\langle L,\neg\rangle$, it follows that $\neg b\in M$.
Moreover, $(\neg b,b)\notin\mathsf{C}_{M}$. Consequently, $M$ is
a strong precontact sublattice of $\langle L,\mathsf{C}_{\neg}\rangle$.

$(2)\Rightarrow(1)$ Suppose that $\langle M,\mathsf{C}_{M}\rangle$
is a strong precontact sublattice of $\langle L,\mathsf{C}_{\neg}\rangle$,
where $\mathsf{C}_{M}=\mathsf{C}_{\neg}\cap(M\times M)$. Let $x\in M$.
We prove that $\neg x\in M$. Since$(\neg x,x)\notin\mathsf{C}_{\neg}$
and $\langle M,\mathsf{C}_{M}\rangle$ is a strong precontact sublattice
of $\langle L,\mathsf{C}_{\neg}\rangle$, there exists $t\in M$ such
that $(t,x)\notin\mathsf{C}_{M}$ and $\neg x\leq t$. As $(t,x)\notin\mathsf{C}_{M}$
and $(t,x)\in M\times M$, it follows that $(t,x)\notin\mathsf{C}_{\neg}$,
since $\mathsf{C}_{M}=\mathsf{C}_{\neg}\cap(M\times M)$. Thus$t\leq\neg x$.
Therefore $\neg x=t\in M$. 
\end{proof}

\section{Congruences of precontact lattices}

\label{sec:Congruences-of-precontact}

In universal algebra, a congruence on an algebra is an equivalence
relation compatible with the algebraic operations. However, when the
algebra is equipped with additional relational structure---such as
a precontact relation---this compatibility alone does not suffice
to ensure that the corresponding quotient structure is well defined.

We now introduce a notion of lattice congruence that is compatible
with the precontact relation, in a sense to be made precise later.
This notion guarantees that the quotient lattice inherits a natural
precontact structure. We also examine the cases in which the precontact
relation is modally definable in the Boolean setting and negationally
definable in the lattice setting.

In the study of the precontact congruences we will take into account
that, if $L$ is a bounded distributive lattice and $\theta\subseteq L\times L$
is a lattice congruence, then $\theta$ can be expressed in terms
of closed subsets of the topology $\mathcal{T}_{X(L)}$. More precisely,
every lattice congruence $\theta\subseteq L\times L$ has associated
a closed subset $Y$ of ${\rm {X}(L)}$ such that $\theta=\theta(Y)$,
where $\theta(Y)$ is defined by: 
\[
(a,b) \in\theta(Y) \iff\varphi(a) \cap Y = \varphi(b) \cap Y.
\]

\begin{definition} Let $\langle L,\mathsf{C}\rangle$ be a precontact
lattice, and let $\theta$ be a lattice congruence on $L$. We define
a relation $\hat{\theta}\subseteq\mathcal{P}(L)\times\mathcal{P}(L)$
by 
\[
\langle U, V\rangle\in\hat{\theta}\iff\text{for every } x \in U \text{ there exists } y \in V \text{ such that } (x,y)\in\theta.
\]
\end{definition}

We proceed to formally introduce the notion of a precontact congruence
in a precontact lattice $\langle L,\mathsf{C}\rangle$. For this purpose,
we use the notation $\alpha(a)=\mathsf{C}^{-1}(a)^{c}$ for each $a\in L$.

\begin{definition} Let $\langle L,\mathsf{C}\rangle$ be a precontact
lattice, and let $\theta$ be a lattice congruence of $L$. We will
say that $\theta$ is a precontact congruence if for every $(a,b)\in\theta$,
it holds that $\langle\alpha(a),\alpha(b)\rangle\in\hat{\theta}$.
\end{definition}

It is important to note that, if formulated in terms of the operator
$\lambda$ instead of $\alpha$, one obtains the dual notion of precontact
congruence. In this work, we adopt the $\alpha$--based formulation
to define and develop all notions related to precontact congruences.
For each result stated using $\alpha$, a dual statement follows by
working with $\lambda$, with analogous proofs. When the relation
$\mathsf{C}$ is symmetric, the two notions coincide.

Given a bounded distributive lattice $L$, it is known that the structure
\[
\mathbf{Con}\,L=\langle\mathbf{Con}\,L,\lor,\land,i_{L},L\times L\rangle
\]
is a lattice, where $\mathbf{Con}\,L$ denotes the family of all lattice
congruences on $L$, $i_{L}$ is the identity congruence on $L$,
and $L\times L$ is the total congruence. Moreover, the quotient structure
$\langle L/\theta,\vee_{\theta},\wedge_{\theta},1_{\theta},0_{\theta}\rangle$
is a bounded distributive lattice, where $L/\theta=\{x_{\theta}:x\in L\}$,
and that the canonical mapping $q\colon L\to L/\theta$ given by $q(a)=a_{\theta}$
is a lattice homomorphism.

Now, if $\langle L,\mathsf{C}\rangle$ is a precontact lattice, let
$\mathbf{Con}_{p}\,L$ denote the set of all precontact congruences
of $\langle L,\mathsf{C}\rangle$. Then it is easy to see that 
\[
\mathbf{Con}_{p}\,L=\langle\mathbf{Con}_{p}\,L,\lor,\land,i_{L},L\times L\rangle
\]
is a sublattice of $\mathbf{Con}\,L$. For each $\theta\in\mathbf{Con}_{p}\,L$,
we define a binary relation $\mathsf{C}_{\theta}$ on $L/\theta$
by 
\[
\left(x_{\theta},y_{\theta}\right)\text{ iff }y_{\theta}\notin\mathrm{Ig}\left(\left\{ a_{\theta}:(x,a)\notin\mathsf{C}\right\} \right).
\]
One can then verify that $\langle L/\theta,\vee_{\theta},\wedge_{\theta},\mathsf{C}_{\theta},1_{\theta},0_{\theta}\rangle$
is a precontact lattice and that the canonical mapping $q\colon\langle L,\mathsf{C}\rangle\to\langle L/\theta,\mathsf{C}_{\theta}\rangle$
is a strong precontact homomorphism.

It is a classical result that the kernel of a lattice homomorphism
determines a lattice congruence. More precisely, if $h\colon L_{1}\to L_{2}$
is a lattice homomorphism, then the relation $\ker(h)=\left\{ (a,b)\in L_{1}\times L_{2}:h(a)=h(b)\right\} $
is a lattice congruence on $L_{1}$. We now show that an analogous
result holds in the context of strong precontact homomorphisms.

\begin{theorem} Let $h\colon(L_{1},\mathsf{C}_{1})\to(L_{2},\mathsf{C}_{2})$
be a strong precontact homomorphism between precontact lattices. Then
the relation $\ker(h)$ is a precontact congruence on $(L_{1},\mathsf{C}_{1})$.
\end{theorem} 
\begin{proof}
Let $h\colon(L_{1},\mathsf{C}_{1})\to(L_{2},\mathsf{C}_{2})$ be a
strong precontact homomorphism between precontact lattices. Let $(a,b)\in\ker(h)$,
i.e., $h(a)=h(b)$, and suppose $x\in\mathsf{C}^{-1}_{1}(a)^{c}$.
Then $(x,a)\notin\mathsf{C}_{1}$, so by the definition of a precontact
homomorphism, $(h(x),h(a))\notin\mathsf{C}_{2}$. Since $h(a)=h(b)$,
it follows that $(h(x),h(b))\notin\mathsf{C}_{2}$. As $h$ is a strong
precontact homomorphism, there exists $t\in\mathsf{C}^{-1}_{1}(b)^{c}$
such that $h(x)\leq h(t)$. Let $y=x\wedge t$. Then $y\in\mathsf{C}^{-1}_{1}(b)^{c}$,
since the set $\mathsf{C}^{-1}_{1}(b)^{c}$ is an ideal, and moreover
$h(y)=h(x\wedge t)=h(x)$. Thus, $(x,y)\in\ker(h)$. This proves that
for any $(a,b)\in\ker(h)$ and any $x\in\mathsf{C}^{-1}_{1}(a)^{c}$,
there exists $y\in\mathsf{C}^{-1}_{1}(b)^{c}$ such that $(x,y)\in\ker(h)$.
Hence $\ker(h)$ is a precontact congruence. 
\end{proof}

Below we will show that the notion of contact congruence generalizes
both the notion of congruence in modal algebras and the notion of
congruence in lattices with negation.

\begin{proposition} Let $L$ be a bounded distributive lattice. Let
$\theta$ be a lattice congruence of $L$. 
\begin{enumerate}
\item If $(L,\square)$ is a modal Boolean algebra and $\theta$ is a $\square$-congruence,
then $\theta$ is a precontact congruence of $(L,\mathsf{C}_{\square})$. 
\item If $(L,\neg)$ is a lattice with a negation operator and $\theta$
is a $\neg$-congruence, then $\theta$ is a precontact congruence
of $(L,\mathsf{C}_{\neg})$. 
\end{enumerate}
\end{proposition} 
\begin{proof}
$(1)$ Suppose that $(L,\square)$ is a modal Boolean algebra and
$\theta$ is a $\square$-congruence. Let $a,b\in B$ be such that
$(a,b)\in\theta$ and $(x,a)\notin\mathsf{C}_{\square}$. By the definition
of $\mathsf{C}_{\square}$, we have $x\leq\square a$. Since $(a,b)\in\theta$
and $\theta$ is a Boolean congruence, $(\square a,\,\square b)\in\theta$.
Also, $(x\wedge\square a,\,x\wedge\square b)=(x,x\wedge\square b)\in\theta$
and $(x\wedge\square b,b)\notin\mathsf{C}_{\square}$. Therefore,
there exists $y=x\wedge\square b$ such that $(y,b)\notin\mathsf{C}_{\square}$
and $(x,y)\in\theta$. Hence $\theta$ is a precontact congruence.

$(2)$ The proof follows the same idea as in $(1)$. We leave it to
the reader. 
\end{proof}

Having introduced the notion of precontact congruence as a lattice
congruence that, in an appropriate sense, preserves the precontact
structure, we now aim to investigate whether this condition admits
a topological characterization

Recall that, under Priestley duality, every lattice congruence $\theta$
is associated with a closed subset $Y\subseteq\mathsf{X}(L)$, in
such a way that $(a,b)\in\theta$ if and only if $\beta(a)\cap Y=\beta(b)\cap Y$.
Given this correspondence, it is natural to ask whether the additional
requirement of precontact compatibility---namely, that  $\langle\alpha(a),\alpha(b)\rangle\in\hat{\theta}$---can
also be understood in topological terms. To this end, recall that
for each ideal $I\subseteq L$, $\varphi(I)=\{P\in\mathsf{X}(L):P\cap I\neq\emptyset\}$
is an open upset in the Priestley space $\mathsf{X}(L)$.

\begin{proposition} Let $\langle L,\mathsf{C}\rangle$ be a precontact
lattice. Let $Y$ be a closed subset of ${\rm {X}(L)}$ such that
$(a,b)\in\theta(Y)$. Then $\varphi(\alpha(a))\cap Y=\varphi(\alpha(b))\cap Y$
iff for each $x\in\alpha(a)$ exists $y\in\alpha(b)$ such that $(x,y)\in\theta(Y)$.
\end{proposition} 
\begin{proof}
Suppose that $\varphi(\alpha(a))\cap Y=\varphi(\alpha(b))\cap Y$.
Let $x\in\alpha(a)$. Thus 
\[
\begin{aligned}\beta(x)\cap Y & \subseteq\bigcup_{x\in\alpha(a)}(\beta(x)\cap Y)=\varphi(\alpha(a))\cap Y=\varphi(\alpha(b))\cap Y\\
 & =\bigcup_{y\in\alpha(b)}(\beta(y)\cap Y)\subseteq\bigcup_{y\in\alpha(b)}\beta(y).
\end{aligned}
\]

Since $\beta(x)\cap Y$ is a closed subset of ${\rm {X}}(L)$, then
by compactness there are $y_{1},\ldots,y_{n}\in\alpha(b)$ such that
$\beta(x)\cap Y\subseteq\beta(y_{1})\cup\ldots\cup\beta(y_{n})$.
So $\beta(x)\cap Y\subseteq\beta(z)\cap Y$, where $z=y_{1}\vee\ldots\vee y_{n}$.
Hence $\beta(x)\cap Y\cap\beta(z)=\beta(x\wedge z)\cap Y=\beta(x)\cap Y$.
Let $y=x\wedge z$. Then $y\in\alpha(b)$, because $y\leq z$ and
$z\in\alpha(b)$ which is an ideal. Therefore, there exists $y\in\alpha(b)$
such that $(x,y)\in\theta(Y)$.

Now suppose that for each $x\in\alpha(a)$ exists $y\in\alpha(b)$
such that $(x,y)\in\theta(Y)$. Let $P\in\varphi(\alpha(a))\cap Y$,
i.e., $\alpha(a)\cap P\neq\emptyset$ and $P\in Y$. Take $x\in\alpha(a)\cap P$.
By hypothesis, there is a $y\in\alpha(b)$ such that $(x,y)\in\theta(Y)$.
Hence $\beta(x)\cap Y=\beta(y)\cap Y$ ; and since $y\in P$, we have
that $P\in\beta(y)\cap Y$. So, $\alpha(b)\cap Y\neq\emptyset$. Thus,
$P\in\varphi(\alpha(b))\cap Y$. The other inclusion follows from
the fact that $(a,b)\in\theta(Y)$ and $\theta(Y)$ is symmetric. 
\end{proof}

\begin{corollary}\label{caract congruencia} Let $\langle L,\mathsf{C}\rangle$
be a precontact lattice, $Y$ a closed subset of $\tau_{{\rm {X}}(L)}$
and let $\theta(Y)$ be its associated lattice congruence. Then, $\theta(Y)$
is a precontact congruence iff for every $(a,b)\in\theta(Y)$, $\varphi(\alpha(a))\cap Y=\varphi(\alpha(b))\cap Y$.

\end{corollary}

We define a special class of closed subsets and show that they correspond
exactly to the congruences of precontact lattices.

\begin{definition} Let $\langle X,R\rangle$ be a precontact space.
A subset $Y\subseteq X$ is called $R$-saturated if $\max R(x)\subseteq Y$
for each $x\in Y$. \end{definition}

It is easy to verify that the intersection of any family of $R$-saturated
sets is also $R$-saturated, and the union of a finite number of them
is $R$-saturated as well. Moreover, since both $X$ and the empty
set are $R$-saturated, the collection of all $R$-saturated subsets
of $X$ forms a complete sublattice of $\mathcal{P}(X)$, which we
will denote by $C_{R}(X)$.

\begin{proposition} Let $\langle L,\mathsf{C}\rangle$ be a precontact
lattice and $\theta\subseteq L\times L$ be a lattice congruence.
The following conditions are equivalent: 
\begin{enumerate}
\item $\langle\alpha(a),\alpha(b)\rangle\in\hat{\theta}(Y)$, for each $(a,b)\in\theta(Y)$. 
\item $Y$ is $R$-saturated. 
\end{enumerate}
\end{proposition} 
\begin{proof}
$(1)\Rightarrow(2)$ Let $P\in Y$ and $Q\in\text{max }R(P)$. Suppose
that $Q\notin Y$. As $Y$ is closed, there exist $a,b\in L$ such
that $Y\subseteq\beta(a)^{c}\cup\beta(b)$ and $Q\notin\beta(a)^{c}\cup\beta(b)$.
This implies that $(a,a\wedge b)\in\theta(Y)$, $a\in Q$, and $b\notin Q$.
On the other hand, it is also easy to see that $\alpha^{-1}(P)\cap\mathrm{Fg}(Q\cup\{b\})\neq\emptyset$.
Consequently, there exist $p\in\alpha^{-1}(P)$ and $q\in Q$ such
that $q\wedge b\leq p$. Since $\alpha^{-1}(P)$ is an ideal and $p\in\alpha^{-1}(P)$,
with $q\wedge b\leq p$, it follows that $a\wedge q\wedge b\in\alpha^{-1}(P)$.
As a result, $P\in\varphi(\alpha(a\wedge q\wedge b))$. Moreover,
$(a\wedge q,a\wedge q\wedge b)\in\theta(Y)$. Applying Corollary \ref{caract congruencia},
we obtain $\varphi(\alpha(a\wedge q))\cap Y=\varphi(\alpha(a\wedge q\wedge b))\cap Y$.
From this, we conclude that $P\in\varphi(\alpha(a\wedge q\wedge b))\cap Y=\varphi(\alpha(a\wedge q))\cap Y$,
which implies that $a\wedge q\in\alpha^{-1}(P)$. Consequently, $a\wedge q\in\alpha^{-1}(P)\cap Q$,
contradicting the fact that $Q\in R(P)$. This completes the proof.

$(2)\Rightarrow(1)$ Suppose that $Y$ is $R$-saturated. Let $(a,b)\in\theta(Y)$.
Assume there exists $P\in\beta(\alpha(a))\cap Y$ such that $P\notin\beta(\alpha(b))\cap Y$.
Then, $\alpha(a)\cap P=\emptyset$ and $\alpha(b)\cap P\neq\emptyset$.
This implies that $a\in\alpha^{-1}(P)$ and $b\notin\alpha^{-1}(P)$.
By Lemma~\ref{separacion}, we may take $Q\in X(L)$ with $b\in Q$
and $Q\in R_{\mathsf{C}}$. As $\text{max }R(P)\neq\emptyset$, there
exists $D\in{\rm {X}}(L)$ such that $(P,D)\in R$ and $Q\subseteq D$.
Since $D\in\beta(b)\cap Y=\beta(a)\cap Y$, $a\in D$. Thus, $\alpha^{-1}(P)\cap D\neq\emptyset$,
which is a contradiction. Therefore, $\beta(a)\cap Y\subseteq\beta(b)\cap Y$.
The other inclusion can be proved in a similar way. 
\end{proof}

\begin{corollary} Let $L$ be a precontact lattice and $\theta\subseteq L\times L$
a lattice congruence. Let $Y\subseteq{\rm {X}(L)}$ be the closed
set associated with $\theta$. Then $\theta(Y)$ is a precontact congruence
iff $Y$ is $R$-saturated. \end{corollary}

\begin{corollary} Let $\langle L,\mathsf{C}\rangle$ be a precontact
lattice. The correspondence $Y\mapsto\theta(Y)$ establishes an anti-isomorphism
between $\mathbf{Con}_{p}\,L$ and $C_{R}(X)$. \end{corollary}

\section{Final remarks and further research directions}

In this work, we develop a duality based on relational Priestley spaces
for the class of distributive lattices equipped with a precontact
relation. This duality is used to give some results on correspondence,
characterize precontact substructures and strong precontact sublattices,
and to introduce and study a suitable notion of congruence. In the
Boolean case, the relational representation allows us to show that
the class of Boolean contact algebras has the joint embedding property
and the amalgamation property \cite{Duntsch4}. In future work, we
plan to explore these properties in the context of distributive lattices
or Heyting algebras with a precontact or contact relation.

For Boolean contact algebras, alternative representation theorems
are available in terms of topological spaces satisfying appropriate
conditions. In particular, \cite{Gold} establishes a duality between
Boolean contact algebras and pairs $\langle X,B\rangle$, where $X$
is a topological space and $B$ is a basis of regular closed sets
forming a subalgebra of $\mathsf{RO}(X)$, the Boolean algebra of
all regular closed subsets of $X$, and satisfying a condition strictly
weaker than compactness. This duality extends earlier representation
results for contact algebras obtained in \cite{Dimov2}, \cite{Duntsch2},
and \cite{Duntsch3}. A natural question is whether an analogous duality
can be established for contact lattices, along the lines of the duality
developed by Goldblatt and Grice in \cite{Gold}. In this direction,
the results presented in \cite{Duntsch1} provide a plausible starting
point.

\section*{Acknowledgements}

We would like to thank the referees for their careful reading of the
manuscript and for their valuable comments and suggestions. Their
insightful remarks have significantly improved the clarity and quality
of the paper.

\end{document}